\documentclass[11pt, reqno]{amsart}
\usepackage{amssymb,latexsym,amsmath,amsfonts,amscd}
\usepackage[mathscr]{eucal}
\numberwithin{equation}{section}

\theoremstyle{plain}
\newtheorem{thm}{Theorem}[section]
\newtheorem{lem}[thm]{Lemma}
\newtheorem{prop}[thm]{Proposition}
\newtheorem{result}[thm]{Result}

\theoremstyle{definition}
\newtheorem{defn}[thm]{Definition}
\newtheorem{fact}[thm]{Fact}

\theoremstyle{remark}
\newtheorem{remark}[thm]{Remark}
\newtheorem{exmpl}[thm]{Example}

\newcommand{\CC}{\mathbb{C}^2}
\newcommand{\CCC}{\mathbb{C}^3}
\newcommand{\Cn}{\mathbb{C}^n}
\newcommand{\cplx}{\mathbb{C}}
\newcommand{\RR}{\mathbb{R}}
\newcommand{\tor}{\mathbb{T}}

\newcommand{\bdy}{\partial}
\newcommand{\OM}{\Omega}

\newcommand\nball[1]{\mathbb{B}^{{#1}}}

\newcommand\scone[2]{\boldsymbol{\mathcal{K}}({#1};{#2})}
\newcommand{\clcone}{\boldsymbol{\overline{\mathcal{K}}}}
\newcommand{\clBcone}{\boldsymbol{\overline{\mathcal{K}}_*}}
\newcommand\smcone[2]{\overline{\boldsymbol{\mathcal{K}}({#1}};{#2})}
\newcommand\vcone[1]{\overline{\boldsymbol{\mathfrak{C}}}_{{#1}}}
\newcommand{\weg}{\boldsymbol{\mathscr{W}}}
\newcommand{\clweg}{\boldsymbol{\overline{\mathscr{W}}}}
\newcommand{\sectr}{\mathcal{S}}

\newcommand{\smoo}{\mathcal{C}}

\newcommand{\psh}{{\sf psh}}

\newcommand{\eps}{\varepsilon}
\newcommand{\al}{\alpha}
\newcommand{\zt}{\zeta}
\newcommand{\zbar}{\overline{z}}

\newcommand{\xibar}{\overline{\xi}}
\newcommand{\tht}{\theta}

\newcommand{\Lam}{\Lambda}
\newcommand{\zahl}{\mathbb{Z}}

\newcommand{\cplxlns}{\boldsymbol{{\sf L}}}
\newcommand{\exepc}{\boldsymbol{\mathcal{E}}}

\newcommand{\nat}{\mathbb{N}}
\newcommand{\tail}{\mathcal{R}}
\newcommand{\mat}{\mathcal{M}}
\newcommand{\bmew}{\boldsymbol{\overline{\eta}}}
\newcommand{\bnew}{\boldsymbol{\overline{\nu}}}

\newcommand\sfrac[2]{\genfrac{}{}{}{1}{#1}{#2}}
\newcommand{\hess}{\mathfrak{H}_{\mathbb{C}}}
\newcommand{\mi}{\mathfrak{Im}}

\newcommand{\er}{\mathfrak{Re}}
\newcommand{\lrarw}{\longrightarrow}

\newcommand\levi[1]{\mathfrak{L}{#1}}

\newcommand{\jayP}{\widetilde{{}^j P}}
\newcommand\rg[1]{{\sf Arg}({#1})}
\newcommand\LeviDeg[1]{\omega({#1})}

\newcommand{\lcm}{{\rm lcm}}
\newcommand{\bcdot}{\boldsymbol{\cdot}}

\newcommand{\nrml}{\boldsymbol{\mathcal{N}}}
\newcommand{\gee}{\widetilde{G}}
\newcommand{\what}{\widehat{w}}
\newcommand{\newR}{\widetilde{\varrho}}
\newcommand{\newRem}{\widetilde{r}}
\newcommand{\Odr}{\boldsymbol{{\sf O}}}

\newcommand\Mixdrv[2]{\partial^2_{{#1}\overline{{#2}}}}
\newcommand\monodrv[1]{\partial_{{#1}}}

\newcommand{\laplc}{\bigtriangleup}

\begin{document}

\title[Model domains and bumping]{Model pseudoconvex domains and bumping} 

\author{Gautam Bharali}
\address{Department of Mathematics, Indian Institute of Science, Bangalore - 560012}
\email{bharali@math.iisc.ernet.in}
\thanks{This work is supported in part by a grant from the UGC under DSA-SAP, Phase~IV}

\begin{abstract}
The Levi geometry at weakly pseudoconvex boundary points of domains in $\Cn, \ n\geq 3$,
is sufficiently complicated that there are no universal model domains with which
to compare a general domain. Good models may be constructed by bumping outward 
a pseudoconvex, finite-type $\OM\subset\CCC$ in such a way that: $i)$ pseudoconvexity
is preserved, $ii)$ the (locally) larger domain has a simpler defining function, and
$iii)$ the lowest possible orders of contact of the bumped domain with $\bdy\OM$,
at the site of the bumping, are realised. When $\OM\subset\cplx^n, \ n\geq 3$, it
is, in general, hard to meet the last two requirements. Such well-controlled bumping
is possible when $\OM$ is $h$-extendible/semiregular. We examine a family of domains
in $\CCC$ that is {\em strictly larger} than the family of $h$-extendible/semiregular
domains and construct explicit models for these domains by bumping.
\end{abstract}

\maketitle

\section{Introduction}\label{S:intro}

A rather successful strategy for understanding a pseudoconvex domain 
in $\Cn, n\geq 2$, involves carefully deforming its boundary about some boundary 
point without destroying pseudoconvexity
so that the new domain ``well approximates'' the original but
has a much simpler defining function. In many instances, one requires the model domain
to be deformed outwards about the chosen boundary point. The latter procedure is formalised as
follows:
\begin{itemize}
 \item[$(*)$] Given a smoothly bounded pseudoconvex domain $\OM\subset\Cn, \ n\geq 2$, 
 and $\zt\in\bdy\OM$, find a neighbourhood $U_\zt$ of $\zt$ and a $\smoo^2$-smooth function
 $\rho_\zt\in\psh(U_\zt)$ such that
 \begin{itemize}
  \item[\textbullet] $\rho_\zt^{-1}\{0\}$ is a smooth hypersurface in $U_\zt$ that is 
  pseudoconvex from the side $U^{-}_\zt:=\{z:\rho_\zt(z)<0\}$; and
  \item[\textbullet] $\rho_\zt(\zt)=0$, but 
  $(\overline{\OM}\setminus\{\zt\})\bigcap U_\zt\varsubsetneq U^{-}_\zt$.
 \end{itemize}
\end{itemize}
We shall call the triple $(\bdy\OM,U_\zt,\rho_\zt)$ a {\em local bumping of $\OM$ about
$\zt$}. Fornaess and Sibony \cite{fornaessSibony:cpfwpd89} devised a procedure of bumping 
to show that every boundary point of a finite-type domain in $\CC$ admits a holomorphic peak function.
The works \cite{fornaess:sneDbC286, range:ikHeDbpdftC290} are just some of the many
applications of bumping in $\CC$.
\smallskip

Going beyond $\CC$, Diederich and Fornaess \cite{diederichFornaess:phmpdrab79} 
have shown that if $\Omega$ is a bounded,
pseudoconvex domain with real-analytic boundary, then local bumpings always exist about each
$\zt\in\bdy\OM$. However, the applications cited above rely on a second ingredient: that 
the bumpings constructed are, in some sense, well-adapted to the pair $(\OM,\zt)$. To be more
specific, it is highly desirable for a local bumping $(\bdy\OM,U_\zt,\rho_\zt)$ to have the
following two properties:
\begin{itemize}
 \item[(B1)] The orders of contact of $\bdy\OM\cap U_\zt$ with
 $\rho_{\zt}^{-1}\{0\}$ at $\zt$ along various directions 
 $V\in T_\zt(\bdy\OM)\bigcap iT_\zt(\bdy\OM)$ are the {\em lowest possible}.
 \item[(B2)] The function $\rho_\zt$ is as simple as possible and is explicitly known.
\end{itemize} 
The difficulty with the results of \cite{diederichFornaess:phmpdrab79} 
is that when $\OM\Subset\Cn$ and $n\geq 3$,
the order of contact between $\bdy\OM$ and $\rho_\zt^{-1}\{0\}$ at $\zt$ along certain
complex-tangential directions can be very high. Furthermore, the great generality of the scope
of \cite{diederichFornaess:phmpdrab79} makes it very hard for an explicit equation for $\rho_\zt$
to be discernable.
\smallskip

Yet, there has been some progress --- which generalises the situation in 
\cite{fornaessSibony:cpfwpd89} --- in achieving $(*)$ so that the local bumping has the
properties (B1) and (B2). To appreciate this, we need to define the Catlin normal form for a 
pair $(\OM,\zt)$. To this end, we refer the reader to Catlin \cite{catlin:bipd84} for a 
definition of the {\em Catlin multitype}.

\begin{defn}\label{D:CatNormForm}
Let $\OM$ be a pseudoconvex domain in $\cplx^{n+1}, \ n\geq 1$, having $\smoo^\infty$-smooth 
boundary and let $\zt\in \bdy\OM$. Let $(1,m_1,\dots,m_n)$ be the Catlin multitype of $\bdy\OM$ at
$\zt$ and suppose $m_j< \infty, \ j=1,\dots, n$. Then, by \cite[Main Theorem]{catlin:bipd84},
there exists a local holomorphic coordinate system $(U_\zt;w,z_1,\dots,z_n)$ centered 
at $\zt$, which we shall call {\em distinguished coordinates}, such that
\begin{equation}\label{E:CatNormForm}
 \OM\bigcap U_\zt \ = \ \left\{(w,z)\in V_\zt:\er{w}+P(z)+Q(z)+r(\mi{w},z)<0\right\},
\end{equation}
where $V_\zt$ is a neighbourhood of $0\in \cplx^{n+1}$, and where
\begin{itemize}
 \item $P$ is an $(m_1,\dots, m_n)$-homogeneous plurisubharmonic polynomial in $\Cn$ 
 that has {\em no pluriharmonic terms};
 \item $D^\alpha\overline{D}^\beta Q(0)=0$ whenever $(\alpha,\beta)\in \nat^n\times \nat^n$
 such that $\sum_{j=1}^n m_j^{-1}(\alpha_j+\beta_j)\leq 1$; and
 \item There exists a smooth function $S$ defined about $0\in \RR$ such that
 \[
  \frac{r(x,z)-S(x)}{|x|} \ = \ o\left(\sqrt{|z_1|^{m_1}+\dots+|z_n|^{m_n}}\right) \;\; \text{as 
  $(x,z)\to 0$.}
 \]
\end{itemize}
Any presentation of $\OM$ in a neighbourhood of $\zt\in\bdy\OM$ having the form \eqref{E:CatNormForm}
will be called a {\em Catlin normal form for $(\OM,\zt)$}.
\end{defn}
\noindent{We remind the reader that, given an $n$-tuple $\Lam=(\lambda_1,\dots,\lambda_n)$, $P$ is said 
to be {\em $\Lam$-homogeneous} if $P(t^{1/\lambda_1}z_1,\dots,t^{1/\lambda_n}z_n)=tP(z_1,\dots,z_n)$
$\forall z=(z_1,\dots,z_n)\in \Cn$ and for every $t>0$. The behaviour of the the term $r$ above is already
implicit in \cite{catlin:bipd84}, but has been made explicit in \cite{yu:wblgK-Rmwpd95} following an argument
by Fornaess--Sibony \cite{fornaessSibony:cpfwpd89}.}
\smallskip

Let the pair $(\OM,\zt)$ be as in Definition~\ref{D:CatNormForm}. If the model domain 
$\OM_P:=\{(w,z)\in \cplx\times\Cn:\er{w}+P(z)<0\}$ --- where $P$ is the polynomial 
occurring in \eqref{E:CatNormForm} --- is of finite type, then it 
was shown independently by Yu in \cite{yu:pfwpd94}, and by Diederich--Herbort in 
\cite{diederichHerbort:pdst94}, that one can make a holomorphic change of the $w$-coordinate,
and find a neighbourhood $V_\zt$ of $0\in \cplx^{n+1}$ and an $(m_1,\dots,m_n)$-homogeneous function 
$H\in\smoo^1(\Cn)$, that is positive away from $0$, such that $(P-H)\in \psh(\Cn)$, and
 \begin{equation}\label{E:DH-Ybump}
 (\overline{\OM}\setminus\{\zt\})\bigcap U_\zt \ \varsubsetneq \ 
	\{(w,z)\in V_\zt:\er{w}+P(z)-H(z)<0\}.
\end{equation}
With $\OM$ as above, whenever the model domain $\OM_P$ associated to its Catlin normal 
form at some $\zt\in \bdy\OM$ is of finite type, we say that $\OM$ 
is {\em $h$-extendible at $\zt\in \bdy\OM$} (or, alternatively, that {\em $(\OM,\zt)$
is $h$-extendible}). It is evident that the bumping represented by the
right-hand side of \eqref{E:DH-Ybump} also has the properties (B1) and (B2). 
Are such elementary model domains available when the pair $(\OM,\zt)$ is
{\em not} $h$-extendible~?
\smallskip

The answer to the above question is definitively, ``No.''  Yu in \cite{yu:pfwpd94} and
Diederich--Herbort in \cite{diederichHerbort:pdst94} have independently shown that:

\begin{fact}\label{F:curves}
An $(m_1,\dots,m_n)$-homogeneous $H$ as in \eqref{E:DH-Ybump} exists if and only if there
are no complex subvarieties of $\Cn$ of positive dimension along which $P$ is
harmonic.
\end{fact}

\noindent{The purpose of this paper is to show that if the domain $\OM$ in question is in $\CCC$ 
and, for the chosen boundary point $\zt$, the structure of the exceptional complex varieties along 
which $P$ --- in the notation of \eqref{E:CatNormForm} --- is not too complicated, then
one can construct a local bumping $(\bdy\OM,U_\zt,\rho_\zt)$ that has the properties (B1) and
(B2) and, in particular, $\rho_\zt$ is not much more complicated than the formula occurring in
\eqref{E:DH-Ybump}. The class of pointed domains $(\OM,\zt)$ we shall discuss is {\em strictly
larger} than the class of $h$-extendible pairs. The precise definition will be given in the
next section, but a representative of the class that we will study is:}
\smallskip

\begin{exmpl}\label{Ex:basic}
\[
\OM \ = \ \left\{(w,z)\in\cplx\times\CC:\er(w)+|z_1|^6|z_2|^2+|z_1|^8+
                \tfrac{15}{7}|z_1|^2\er(z_1^6)+|z_2|^{10}<0\right\}.
\]
\end{exmpl}
\smallskip

\noindent{Note that $(\OM,0)$ is {\em not} an $h$-extendible pair because 
$P(z_1,z_2)=|z_1|^6|z_2|^2+|z_1|^8+\tfrac{15}{7}|z_1|^2\er(z_1^6)$ (the above example is
already in Catlin normal form at $\zt=0$) is harmonic along the complex line 
$\{(z_1,z_2)\in\CC:z_1=0\}$. The purpose
of this paper is to construct model bumpings for domains of the class to which Example~\ref{Ex:basic}
belongs.}  
\smallskip

\section{Statement of results}\label{S:results}

We begin with some notation that is relevant to the definitions and results in this section.
Let $P$ be an $(m_1,m_2)$-homogeneous plurisubharmonic polynomial on $\CC$, $m_1,m_2\in \zahl_+$,
and define:
\begin{align}
 \LeviDeg{P}  \ :=& \ \{z\in\CC:\hess(P)(z) \ \text{is not strictly positive definite}\},\notag \\
 \mathfrak{C}(P) \ :=& \ \text{the set of all irreducible complex curves $X\subset\CC$} \notag \\
 & \ \text{such that $P$ is harmonic along the smooth part of $X$}, \notag \\
 \exepc(P) \ :=& \ \text{the class of all curves of the form} \notag \\
 &\qquad \left\{(z_1,z_2):z_1^{m_1/\gcd(m_1,m_2)}=\xi z_2^{m_2/\gcd(m_1,m_2)}\right\}, \notag \\
 & \ \text{$\xi\in \widehat{\cplx {\;}}$, along which $P$ is harmonic}\notag
\end{align}
(understanding that $\xi=\infty\Rightarrow$ $P$ is harmonic along
$\{(z_1,z_2)\in\CC:z_2=0\}$), where $\hess(P)(z)$ denotes the complex Hessian of $P$ at $z\in\CC$.
Here $P$ is the prototype for the lowest-weight polynomial that occurs in a Catlin normal form; see
\eqref{E:CatNormForm}. The set $\exepc(P)$ will be the focus of our attention. This is because 
whenever $\mathfrak{C}(P)\neq \varnothing$, then $\exepc(P)\neq \varnothing$. This follows
by making the obvious modifications --- to allow for the fact that $P$ is 
$(m_1,m_2)$-homogeneous --- to the proof of the following 
observation by Noell:

\begin{result}[Lemma~4.2, \cite{Noell:pfpd93}] Let $P$ be a homogeneous, plurisubharmonic, non-pluriharmonic
polynomial in $\Cn, \ n\geq 2$. Suppose there exist complex-analytic varieties of positive dimension in $\Cn$
along which $P$ is harmonic. Then, there exist complex lines through the origin in $\Cn$ along which $P$
is harmonic.
\end{result}

Note that, in view of Fact~\ref{F:curves}, if a pair $(\OM,\zt)$ is {\em not} $h$-extendible, then 
$\exepc(P)\neq \varnothing$. Our claim is that, {\em loosely speaking}, the Levi-geometry around a 
boundary point $\zt$ is tractable enough to enable simple bumpings about $\zt$ for those
non-$h$-extendible pairs $(\OM,\zt)$ for which a Catlin normal form has the property
that $\mathfrak{C}(P)$ contains no complex curves other than those in $\exepc(P)$. 
However, we need to make this assertion precise. To do so, we need the following object. 
An {\em $(m_1,m_2)$-wedge in $\CC$} is defined to be a set $\weg$ having the 
property that if $(z_1,z_2)\in\weg$, then $(t^{1/m_1}z_1,t^{1/m_2}z_2)\in\weg \ \forall t>0$. The 
terms {\em open $(m_1,m_2)$-wedge} and {\em closed $(m_1,m_2)$-wedge} will have the usual meanings.
Note that when $m_1=m_2=2k$ (the true homogeneous case), an $(m_1,m_2)$-wedge is simply a cone. We
also clarify that, in what follows, ``finite type'' will be used in the sense of D'Angelo. I.e., if
$\bdy\OM$ is of finite type at $\zt\in \bdy\OM$, we will mean that the $1$-type 
$\varDelta_1(\bdy\OM,\zt)$ --- see \cite{d'angelo:rhoca82} for a definition --- is finite.   
We can now present our key definition:

\begin{defn}\label{D:almost-h-ext}
Let $\OM$ be a pseudoconvex domain in $\CCC$ having $\smoo^\infty$-smooth
boundary and let $\zt\in \bdy\OM$. Let $(1,m_1,m_2)$ be the Catlin multitype of $\bdy\OM$ at
$\zt$. We say that $\OM$ is {\em almost $h$-extendible at $\zt\in \bdy\OM$}
(or, alternatively, that {\em $(\OM,\zt)$ is almost $h$-extendible}) if $\bdy\OM$ is of finite
type at $\zt$ and, for some (consequently, for every) system of distinguished coordinates
for $(\OM,\zt)$, the lowest-weight plurisubharmonic polynomial $P$ occurring in the 
normal form \eqref{E:CatNormForm} has the following property:
\begin{center}
\begin{tabular}{p{12.7cm}}
 $\LeviDeg{P}\setminus\bigcup_{X\in\exepc(P)}X$ contains
 no complex subvarieties of positive dimension and is {\em well separated from}
 $\bigcup_{X\in\exepc(P)}X$, i.e., there is a closed $(m_1,m_2)$-wedge
 $\clweg$ that contains $\LeviDeg{P}\setminus\bigcup_{X\in\exepc(P)}X$ and
 satisfies $\clweg\bigcap(\bigcup_{X\in\exepc(P)}X)=\{0\}$.
\end{tabular}
\end{center}
\end{defn}

\begin{remark} It is evident from Definition~\ref{D:CatNormForm} that distinguished 
coordinates\linebreak
$(U_\zt; w,z_1,z_2)$ need not be unique. However, it is easy to see that for each $P$ 
(i.e., relative to each system of distinguished coordinates), whether $\exepc(P)=\varnothing$ or 
$\exepc(P)\neq \varnothing$, and whether or not $\LeviDeg{P}\setminus\bigcup_{X\in\exepc(P)}X$ is well
separated from $\bigcup_{X\in\exepc(P)}X$ in the latter case, is independent of the choice
of distinguished coordinates. Furthermore, $(m_1,m_2)\in \zahl_+\times\zahl_+$. This is {\em not} 
evident from the definition of the Catlin multitype, but is a consequence of 
\cite[Theorem~2.2]{catlin:bipd84}.
\end{remark}
\smallskip

\begin{remark} Observe that if $(\OM,\zt)$ is $h$-extendible (in which case $\exepc(P)=\varnothing$),
then it is almost $h$-extendible.
\end{remark}
\smallskip

We must record one notational point: given a system of distinguished coordinates 
$(U_\zt;w,z_1,\dots,z_n)$, $w$ and $z_j, \ j=1,\dots,n$, will {\em interchangably denote}
the standard coordinates of $\cplx\times\Cn$ in which the normal form is presented as well
as the components of an injective holomorphic map
$(w,z_1,\dots,z_n):(U_\zt,\zt)\lrarw (\cplx^{n+1},0)$. The pairs $(\OM,\zt)$ for which we
can state our main theorem will be required to satisfy the following mild condition:
\begin{itemize}
\item[$(**)$] For each $X\in\exepc(P)$ and each $x\in (X\setminus\{(0,0)\})\bigcap{\sf Dom}(Q)$
\[
 Ord\left(\left. P\right|_{\nrml^x}-P(x),x\right) \ \leq \ 
	Ord\left(\left. Q\right|_{\nrml^x}-Q(x),x\right),
\]
and, if $R>0$ is so small that $\overline{\nball{3}(0;R)}\subset (w,z_1,z_2)(U_\zt)$, then
\[
 |r(\mi{w},z)| \ \lesssim \ |\mi{w}|^{\varDelta_1(\bdy\OM)+\sfrac{1}{\lcm(m_1,m_2)}}
 \;\; \forall(w,z)\in \nball{3}(0;R),
\]  
where $(m_1,m_2)$ and $(P,Q,r)$ are determined by the distinguished coordinates 
$(U_\zt;w,z_1,z_2)$, and $\nrml^x$ denotes the complex line in $\CC_z$ that is normal to the 
curve $X$ at $x$. We have abbreviated the $1$-type $\varDelta_1(\bdy\OM,\zt)$ to 
$\varDelta_1(\bdy\OM)$. 
\end{itemize}
The condition $(**)$ might seem like a severe restriction on the class of almost $h$-extendible
pairs, so a couple of explanations are in order. Given a 
point $a\in \cplx$ and a smooth function $F$ defined in a neighbourhood of $a$ such that
$F(a)=0$, we write:
\begin{multline}
Ord(F,a) \ := \ \min\left\{n\in \zahl_+: \text{{\em some} monomial of degree $n$ 
		in the Taylor}\right. \\
		\left.\text{expansion of $F$ around $a$ has non-zero coefficient}\right\}. \notag
\end{multline}
Now, concerning the first part of $(**)$: $(\left. P\right|_{\nrml^x}-P(x))$ may well vanish to 
infinite order in some {\em real} direction in $\nrml^x$, but we do {\em not} require
$(\left. Q\right|_{\nrml^x}-Q(x))$ to have the same decay along such a direction if
$Ord(\left. P\right|_{\nrml^x}-P(x))<\infty$. (Refer also to the short remark just after 
Theorem~\ref{T:main}.) As for $(**)$: {\em we do not require the condition on $r$} in order to 
obtain a result in spirit of Theorem~\ref{T:main}. But there is a pragmatic reason for imposing this 
condition; the reader is referred to Remark~\ref{Rem:long} following Theorem~\ref{T:main}.
\smallskip

We can now state the main theorem of this paper.

\begin{thm}\label{T:main}
Let $\OM$ be a pseudoconvex domain in $\CCC$ having $\smoo^\infty$-smooth
boundary and let $\zt\in \bdy\OM$. Assume that $\OM$ is almost $h$-extendible at $\zt$. Let
$(1,m_1,m_2)$ be the Catlin multitype of $\bdy\OM$ at $\zt$. Let
\[
 \OM\bigcap U_\zt \ = \ \left\{(w,z)\in V_\zt:\er{w}+P(z)+Q(z)+r(\mi{w},z)<0\right\}
\]
be the local representation of $\OM$ with respect to a system of distinguished coordinates
$(U_\zt;w,z_1,z_2)$ (here, and below, $z:=(z_1,z_2)$). If $\exepc(P)\neq \varnothing$, then:
\begin{enumerate}
 \item[1)] $\exepc(P)$ is a finite set, denoted as $\{X_1,\dots,X_n\}$.

 \item[2)] Suppose $(\OM,\zt)$ satisfies $(**)$. Then, we can find a system of holomorphic
 coordinates $(U_\zt;W,Z_1,Z_2)$ in such a way that $Z=z$, and construct a plurisubharmonic function 
 $G\in \smoo^\infty(\CC\setminus\{0\})\bigcap \smoo^2(\CC)$ whose orders of vanishing at $0\in \CC$
 along various complex directions in $\CC$ are explicitly known and
 such that
 \begin{equation}\label{E:bumped}
 (\overline\OM\setminus \{0\})\bigcap\nball{3}(0;R)  \varsubsetneq 
 \{(W,z)\in \CCC: \er{W}+G(z)<0\}\bigcap\nball{3}(0;R)
 \end{equation}
 for some $R>0$.
\end{enumerate}
The function $G$ has the following description. There exist a non-negative function
$\mathcal{H}_0\in \smoo^\infty(\CC\setminus\{0\})\bigcap \smoo^2(\CC)$ that is
$(m_1,m_2)$-homogeneous; subharmonic functions 
$v_j\in \smoo^\infty(\cplx\setminus\{0\})\bigcap \smoo^2(\cplx)$ that are 
homogeneous of degree $2d_j, \ j=1,\dots, N$, and are strictly subharmonic away 
from $0\in \cplx$; and closed $(m_1,m_2)$-wedges $\clweg^1_j$ and $\clweg^2_j$ satisfying
\[
 \clweg^1_j\setminus\{0\} \varsubsetneq {\rm int}(\clweg^2_j), \;\; j=1,\dots,N,
\]
such that:
\begin{itemize}
 \item $\mathcal{H}_0^{-1}\{0\}=\bigcup_{j=1}^N X_j$.
 \item For each $j\leq N$, $G(z) = (P-\mathcal{H}_0)(z)+v_j(z_{k(j)}) \;\; \forall (z_1,z_2)\in 
	\clweg^1_j\bigcap \nball{2}(0;R)$, where $k(j)=1 \ \text{or} \ 2$ depending on $j=1,\dots, N$.
 \item $G(z_1,z_2)=(P-\mathcal{H}_0)(z_1,z_2) \;\; 
	\forall (z_1,z_2)\in \left(\CC\setminus \cup_{j=1}^N\clweg^2_j\right)\bigcap \nball{2}(0;R)$.
\end{itemize}
\end{thm}

\begin{remark}
We have given a part of the explanation for the assertion that
the first part of condition $(**)$ is mild. Another  reason for this assertion is that
that condition can be dispensed with. I.e., {\em even without that condition}, a bumping theorem
similar to Theorem~\ref{T:main} can be deduced. However, the model that one would get would have a more
complicated defining function than the one above. Since the main point of Theorem~\ref{T:main} is to
construct the simplest bumpings in the almost $h$-extendible case, we will assume that the the first
part of $(**)$ is in effect.
\end{remark}
\smallskip

\begin{remark}\label{Rem:long}
The second part of the condition $(**)$ can also be dispensed with. The theorem that one obtains
without this assumption differs from Theorem~\ref{T:main} in the following respect: in addition
to all the objects asserted to exist in Theorem~\ref{T:main}, there exist polynomials
$q_j\in \cplx[z_1,z_2]$, $j=1,\dots,N$, with $deg(q_j)<2d_j$ such that
\[
 G(z) = (P-\mathcal{H}_0)(z)+\er{q_j}(z)+v_j(z_{k(j)}) \;\; \forall (z_1,z_2)\in
        \clweg^1_j\bigcap \nball{2}(0;R),
\]
where $k(j)=1 \ \text{or} \ 2$ depending on $j=1,\dots, N$. This latter version takes some extra effort
to prove but has the advantage that, for domains in $\CCC$, it subsumes the bumping results in 
\cite{diederichHerbort:pdst94} and \cite{yu:pfwpd94}. However, it is not clear whether the right-hand
side of the above equation is simple enough to be useful in applications. In contrast, if we study
{\em unbounded} domains of the form \eqref{E:CatNormForm} with $Q\in \smoo^\infty(\CC)$ and 
$r\equiv 0$, then Theorem~\ref{T:main} can be applied to several problems on such domains. For example,
using Theorem~\ref{T:main} one can extend (in $\CCC$) \cite[Theorem~1.2]{chnKamimtOhsawa:bBki04} by
Chen--Kamimoto--Ohsawa to a much wider class of unbounded domains (which we shall address in a
forthcoming work). To summarise: because of the utility of the bumpings we get from Theorem~\ref{T:main},
we prefer to restrict attention to domains that satisfy the second part of $(**)$.
\end{remark}
\smallskip

We ought to mention that the idea of ``almost $h$-extendibility'' comes from 
\cite{bharaliStensones:psb09}. The property that the polynomial $P$ occurring in 
Definition~\ref{D:almost-h-ext} possesses is termed {\em Property~(A)} in 
\cite{bharaliStensones:psb09}. In fact, some of the preliminary steps in our proof 
of Theorem~\ref{T:main} are variations on the work in \cite{bharaliStensones:psb09}.
Before we conclude this section, let us glance at the key ideas involved in the proof of 
Theorem~\ref{T:main}. The four primary
ingredients in our proof are as follows:
\begin{itemize}
 \item {\em Step 1:} As in \cite{bharaliStensones:psb09}, we begin by studying the true 
 homogeneous case, i.e. when the Catlin multitype of $\bdy\OM$ at $\zt$ is $(1,2k,2k), \ k\geq 2$.
 For each complex line $X\in \exepc(P)$, we can find a closed cone $\clcone^X$ and a function 
 $H_X$ that is supported in $\clcone^X$ such that  
 $(P-\delta H_X)$ is a bumping of $P$, for $\delta>0$ sufficiently small, inside the
 aforementioned cone. The departure from \cite{bharaliStensones:psb09} here is that
 $H_X$ must be constructed with greater delicacy so that the precise orders of decay of $H_X$,
 as one approaches its zero set, are known. This involves some new ideas.
 \smallskip

 \item {\em Step 2:} We know from \cite[Proposition ~1]{bharaliStensones:psb09} that, in the 
 homogeneous case, there are finitely many complex lines $X_1,\dots, X_N\in \exepc(P)$. We use 
 pseudoconvexity of $\OM$, plus the
 hypothesis of $\bdy\OM$ being of finite type at $\zt$, to extract summands from the
 Taylor expansion of $(P+Q)$ (see Definition~\ref{D:CatNormForm}) that 
 constitute homogeneous subharmonic polynomials corresponding to each $X_j, \ j=1,\dots,N$. 
 A well-known prescription is used to bump these polynomials, which yields
 the functions $v_1,\dots, v_N$ of Theorem~\ref{T:main}.
 \smallskip

 \item {\em Step 3:} The properties of $P$ allow us to patch together all the functions constructed
 up to this point in an appropriate manner. This yields the function $G$ of Theorem~\ref{T:main}.
 We then account for those terms in the Taylor expansion
 of $(P+Q)$ that are not already in $\left. G\right|_{\clcone^X}$ so that we can eventually
 conclude that $G(z)<(P+Q)(z) \ \forall z\in \nball{2}(0;R)\setminus\{(0,0)\}$ for some
 $R>0$ sufficiently small.
 \smallskip

 \item {\em Step 4:} The analogous results for the case when the Catlin multitype of $\bdy\OM$ at 
 $\zt$ is $(1,m_1,m_2)$, $m_1\neq m_2$, are obtained by applying an appropriate proper holomorphic
 map that pulls back the Catlin normal form of $(\OM,\zt)$ to a domain to which the 
 above ideas can be applied. A final change of coordinate accounts for the terms
 in the Taylor expansion of $r$ (see Definition~\ref{D:CatNormForm}).
\end{itemize}
\smallskip

The constructions described in the first step of the above procedure will be explained in
Section~\ref{S:leadBump}. The second and the third steps represent the key technical 
proposition of this paper. They constitute Proposition~\ref{P:homoBump}, which will be
presented in Section~\ref{S:homoBump} below. The final part of the above analysis will
be found in Section~\ref{S:main}
\medskip

\section{Technical preliminaries}\label{S:tech}

The purpose of this section is to list some results that we shall apply many times,
in the sections that follow, to accomplish specific technical tasks.
\smallskip

We begin by stating a proposition that is essential for the proof of Theorem~\ref{T:main}, and
from which Part~(1) of Theorem~\ref{T:main} follows almost immediately.

\begin{result}[Proposition~1, \cite{bharaliStensones:psb09}]\label{R:harmlines} 
Let $P$ be a plurisubharmonic, non-pluriharmonic polynomial in $\CC$ that is homogeneous of degree 
$2k$. Then, there are at most finitely many complex lines passing through $0\in\CC$ along which $P$ is harmonic.
\end{result}
\smallskip

The next result is a mild refinement of \cite[Lemma~2.4]{fornaessSibony:cpfwpd89} by 
Fornaess \& Sibony. In fact, we do not need to alter anything in the 
proof of \cite[Lemma~2.4]{fornaessSibony:cpfwpd89}. The refinement lies in stating explicitly, 
in Part (c) below, a property of the Fornaess--Sibony construction of which there was no need in
\cite{fornaessSibony:cpfwpd89}, but which we will need in Section~\ref{S:leadBump}. 

\begin{result}\label{R:subhBump} Let $U:\cplx\to\RR$ be a real-analytic,
subharmonic, non-harmonic function that is homogeneous of degree $j$. Assume that
\[
 \{\theta_1,\dots,\theta_M\} := \{\theta\in[0,2\pi):\laplc U(e^{i\theta})=0\} \subset (0,2\pi).
\]
Let $\sigma_0>0$ be so small that 
$[\theta_k-\sigma_0,\theta_k+\sigma_0]\subset (0,2\pi)$ and 
$[\theta_k-\sigma_0,\theta_k+\sigma_0]\bigcap [\theta_{k+1}-\sigma_0,\theta_{k+1}+\sigma_0]=\varnothing,
\ k=1,\dots M-1$. Next, define:
\[
 \sectr_k(\sigma) \ := \ \{re^{i\theta}: r\geq 0, \ \theta\in [\theta_k-\sigma,\theta_k+\sigma]\}, 
 \;\; \sigma\in (0,\sigma_0).
\]
There exist positive constants
\begin{itemize}
 \item $C_1\equiv C_1(U)$, which varies polynomially in $\sup_{|z|=1}|U(z)|$;
 \item $C_2\equiv C_2(U,\sigma)$, which varies polynomially in $\sup_{|z|=1}|U(z)|$ and 
 $\sigma\in (0,\sigma_0)$;
\end{itemize}
and a $2\pi$-periodic function $h\in\smoo^\infty(\RR)$ such that:
\begin{enumerate}
 \item[$a)$] $0<h(x)\leq 1 \ \forall x\in\RR$.
 \item[$b)$] $\laplc\left(U-\delta|\bcdot|^j h\circ\rg{\bcdot}\right)(z)
 \geq \delta C_1|z|^{j-2} \ \forall z\in\cplx$
 and $\forall\delta:0<\delta\leq 1$.
 \item[$c)$] $\laplc\left(U-\delta|\bcdot|^j h\circ\rg{\bcdot}\right)(z)
 \geq C_2|z|^{j-2} \ \forall z\in \cplx\setminus\left(\bigcup_{k=1}^M\sectr_k(\sigma)\right)$
 and $\forall\sigma\in (0,\sigma_0)$, independent of $\delta$.
\end{enumerate}
(Here $\rg{\bcdot}$ refers to any continuous branch of the argument.)
\end{result}

\begin{remark}
The $\delta>0$ appearing in the above lemma must not be confused with the $\delta$ appearing
in the statement \cite[Lemma~2.4]{fornaessSibony:cpfwpd89}. The latter $\delta$ is a universal
constant which is a component of the constant $C_1(U)$ in our notation. If we denote the
$\delta$ of \cite[Lemma~2.4]{fornaessSibony:cpfwpd89} by $\delta_{{\rm univ}}$, then our
$C_1(U)$ is a polynomial function of $\delta_{{\rm univ}}$ and
\[
\text{(in the notation of \cite{fornaessSibony:cpfwpd89})} \;\; \|U\| \ := \
\sup_{|z|=1}|U(z)|.
\]
\end{remark}
\smallskip

The last technical item is a Levi-form calculation. We will need this result in 
Section~\ref{S:leadBump}.

\begin{lem}\label{L:loplush} 
Let $P(z_1,z_2)$ be a plurisubharmonic, non-pluriharmonic polynomial that is homogeneous of degree $2k$,
$k\geq 2$, and assume that $P(0,\bcdot)\equiv 0$.
Write   
\[
P(z_1,z_2) \ = \ \sum_{j=\mu}^{2k}Q_j(z_1,z_2),
\]
where each $Q_j$ is the sum of all monomials of $P$ that involve powers of $z_1$ and
$\zbar_1$ having total degree $j, \ \mu\leq j\leq 2k$. Then $Q_\mu$ is plurisubharmonic.
\end{lem}
\begin{proof}
We start with a Levi-form calculation. For $\zt, w\in \cplx$, let us write
\begin{align}
 w \ &= \ |w|e^{i\phi}, \notag \\
 \zt \ &= \ |\zt|e^{i\alpha}, \;\; \text{where $\phi,\alpha\in \tor:=\RR/2\pi\zahl$.}\notag
\end{align}
With this notation, the Levi-form of $P$ at the points $(\zt w,w)$ can be written as
\begin{multline}\label{E:leviform}
\levi{P}(\zt w,w;v) = \ |w|^{2(k-1)}|\zt|^{\mu-2}\times(v_1 \quad \zt v_2)
                                \begin{pmatrix}
                                \ T_{11}(\phi,\alpha) & T_{12}(\phi,\alpha) \ \\
                                {} & {} \\
                                \ \overline{T_{12}(\phi,\alpha)} & T_{22}(\phi,\alpha) \
                                \end{pmatrix} \begin{pmatrix}
                                                \ \overline{v_1} \ \\ 
                                                {} \\
                                                \ \overline{\zt v_2} \
                                                \end{pmatrix} \\
	 + \ |w|^{2(k-1)}O( \ |\zt|^{\mu-1}|v_1|^2, |\zt|^\mu|v_1v_2|,|\zt|^{\mu+1}|v_2|^2) \ \geq \ 0,
\end{multline}
where $T_{11}$, $T_{12}$ and $T_{22}$ are trigonometric polynomials obtained when
$\levi{Q_\mu}(\zt w,w;v)$ is written out relative to the polar coordinates defined above.
We first note that as $P\in \psh(\CC)$, whence $\mu\geq 2$, 
$\levi{Q_\mu}((0,\bcdot); \ \bcdot)\geq 0$. Let us assume that $Q_\mu\notin \psh(\CC)$.
Then, there exist $\zt_0\neq 0$, $w_0\neq 0$ and a vector $V=(V_1,V_2)\in\CC$ such that
\[
\levi{Q_\mu}(\zt_0w_0,w_0;V) \ = \ -C \ < \ 0.
\]
Now, for each $r>0$, define $W_r:=(rV_1,V_2)$. From the expression for the Levi-form of $Q_\mu$
in \eqref{E:leviform} above, it clear that
\begin{equation}\label{E:homoed}
\levi{Q_\mu}(r\zt_0 w_0,w_0;W_r) \ = \ r^\mu\levi{Q_\mu}(\zt_0 w_0,w_0;V) \ = \ -Cr^\mu.
\end{equation}
Notice that by the definition of $Q_l$,
\[
\Mixdrv{j}{k}Q_l(r\zt_0w_0,w_0)W_{r,j}\overline{W_{r,k}} \ = \ O(r^{l}) \quad \forall l\geq \mu+1,
                \ j,k=1,2.
\]
Combining this fact with \eqref{E:homoed}, we see that there exists a positive constant
$\delta\ll 1$ such that
\[
\levi{P}(r\zt_0 w_0,w_0;W_r) \ = \ -Cr^\mu + O(r^{\mu+1}) \ < \ 0 \quad
\forall r\in(0,\delta).
\]
But this contradicts the plurisubharmonicity of $P$. Hence our earlier assumption must
be false, and $Q_\mu$ is plurisubharmonic.
\end{proof}
\smallskip

\section{Homogeneous polynomials: bumping around $\exepc(P)$}\label{S:leadBump}
 
This section is devoted to making precise the ideas presented in Step~1 of the outline given in
Section~\ref{S:results}. In order to state the main result of this section cleanly, we need
the following notation: for any $\xi\in\cplx$, $\scone{\xi}{\eps}$ will denote the open cone
\[
\scone{\xi}{\eps} \ := \ \{(z_1,z_2)\in\CC:|z_1-\xi z_2|<\eps|z_2|\}.
\]
Note that $\scone{\xi}{\eps}$ is a conical neighbourhood of the punctured complex line
$\{(z_1=\xi z_2,z_2):z_2\in\cplx\setminus\{0\}\}$. As stated in Section~\ref{S:results},
the initial ideas in the proof of the proposition below are similar to those in the proof
of \cite[Proposition~2]{bharaliStensones:psb09}. The departure from \cite{bharaliStensones:psb09} 
here is represented by the estimate in Part~$(a)$ of Proposition~\ref{P:leadBump} below.
This estimate is a more precise statement than Part~$(a)$ of \cite[Proposition~2]{bharaliStensones:psb09}.
Establishing this requires a more delicate construction, which relies, in part, on the following result:
\smallskip

\begin{result}[Theorem~3, \cite{bharaliStensones:psb09}]\label{R:2homo}
Let $Q(z_1,z_2)$ be a plurisubharmonic, non-harmonic polynomial
that is homogeneous of degree $2p$ in $z_1$ and $2q$ in $z_2$. Then, $Q$ is of the form
\[
Q(z_1,z_2) \ = \ U(z_1^d z_2^D),
\]
where $d,D\in\zahl_+$ and $U$ is a homogeneous, subharmonic, non-harmonic polynomial.
\end{result}
\smallskip

A comment about the hypothesis imposed on $P$ in the result below: the $P$ below is the
prototype for the polynomial $P$ described in Definition~\ref{D:CatNormForm} under the following
assumptions:
\begin{itemize}
 \item The pair $(\OM,\zt)$ is such that the Catlin multitype of $\bdy\OM$ at $\zt$ is 
 $(2k,2k)$.
 \item $\OM$ is almost $h$-extendible at $\zt$ with $\exepc(P)\neq \varnothing$.
\end{itemize}
These are encoded in the hypothesis about $\left. P\right|_{\scone{0}{\eps}\setminus L}$.
In this prototype, $L$ must be viewed as a complex line belonging to $\exepc(P)$.

\begin{prop}\label{P:leadBump} Let $P(z_1,z_2)$ be a plurisubharmonic polynomial that is
homogeneous of degree $2k$, $k\geq 2$, contains no pluriharmonic terms, and such that
\[
 P(z_1,z_2) \ = \ \sum_{j=\mu}^{2k}Q_j(z_1,z_2), \;\; \mu>0,
\]
where each $Q_j$ is the sum of all monomials of $P$ that involve powers of $z_1$ and
$\zbar_1$ having total degree $j$. Write $L:=\{(z_1,z_2):z_1=0\}$, and assume that
there exists an $\eps>0$ such that $P$ is strictly plurisubharmonic in the cone 
$\left(\scone{0}{\eps}\setminus L\right)$.
Then, there exist constants $c,\sigma>0$ --- which depend only on $P$ --- 
and a non-negative function $H\in\smoo^\infty(\CC\setminus\{0\})\bigcap \smoo^2(\CC)$
that is homogeneous of degree $2k$ such that:
\begin{enumerate}
 \item[$a)$] $H(z_1,z_2)>c|z_1|^\mu|z_2|^{2k-\mu}$ when $0<|z_1|<\sigma|z_2|$.
 \item[$b)$] For every $\delta:0<\delta\leq 1$, $(P-\delta H)$ is strictly
 plurisubharmonic on $\left(\scone{0}{\sigma}\setminus L\right)$. Furthermore,
 there is a constant $B\equiv B(t)>0$ that is independent of $\delta\in (0,1]$ --- and 
 depends only on $\sigma$ and $t$ --- such that, whenever $t\in (0,1)$, 
 \begin{align}
 \levi{(P-\delta H)}(z;(V_1,V_2)) \ \geq \ &B(t)\|z\|^{2(k-1)}\|V\|^2 \notag \\ 
	&\forall z: t\sigma|z_2|<|z_1|<\sigma|z_2|, \; \; \forall V\in\CC,\notag
 \end{align}
 for every $\delta:0<\delta\leq 1$.
\end{enumerate}
\end{prop}
\begin{proof}
From Result~\ref{R:2homo}, we deduce that there exist a homogeneous, subharmonic, non-harmonic
polynomial $U$; and $a, b\in \zahl_+$ such that $Q_\mu(z_1,z_2)=U(z_1^a z_2^b)$. Then, adopting
the notation used in the proof of Lemma~\ref{L:loplush} above, we see that the Levi-form of $P$ at the 
points $(\zt w,w)$ has the form:
\begin{align}
 &\levi{P}(\zt w,w;V) \ = \ U_{\xi\xibar}(\zt^a w^{a+b})J(\zt,w)|aV_1+b\zt V_2|^2 \notag \\
 &\; \; + \ |w|^{2(k-1)}\sum_{j=\mu+1}^{2k}|\zt|^{j-2}\times(V_1 \quad \zt V_2)
                                \begin{pmatrix}
                                \ T_{11}^{j}(\phi,\alpha) & T_{12}^{j}(\phi,\alpha) \ \\
                                {} & {} \\
                                \ \overline{T_{12}^{j}(\phi,\alpha)} & T_{22}^{j}(\phi,\alpha) \
                                \end{pmatrix} \begin{pmatrix}
						\ \overline{V_1} \ \\
						{} \\
						\ \overline{\zt V_2} \
						\end{pmatrix} \ 
						\notag \\
 &\; \; \equiv \ U_{\xi\xibar}(\zt^a w^{a+b})J(\zt,w)|aV_1+b\zt V_2|^2
		+ |w|^{2(k-1)}\tail(\zt,\phi;V), \label{E:leviformP}
\end{align}
where $J(\zt,w):=|\zt|^{2(a-1)}|w|^{2(a+b-1)}$; $T_{11}^{j}$, $T_{12}^{j}$ and 
$T_{22}^{j}$ are trigonometric polynomials obtained
when $\levi{Q_j}(\zt w,w;V)$ is written out in polar coordinates as described in the proof of
Lemma~\ref{L:loplush}; and where $\tail:\cplx\times\tor\times\CC\lrarw \RR$. Define the
matrix-valued function $\mat:\cplx\times\tor\lrarw \cplx^{2\times 2}$ by the
relation
\[
 \tail(\zt,\phi;V) \ = \ \langle \mat(\zt,\phi)(V_1,\zt V_2),(V_1,\zt V_2)\rangle,
\]
where the inner product above is the standard Hermitian inner product on $\CC$. We wish to
construct the $H$ mentioned in the statement of this proposition by applying 
Result~\ref{R:subhBump} to the polynomial $U$. Consequently, establishing Part~$(b)$ of this 
proposition would require some estimates on the quantity $\tail$. We shall begin with this 
task and  divide our proof into three steps.
\smallskip
   
\noindent{{\bf Step~1.} {\em Positivity estimates on $\tail$ when $V$ lies in a 
conical neighbourhood of\linebreak
$\{V\in \CC:aV_1+b\zt V_2=0\}$}}

\noindent{Consider the set
\[
 \mathcal{S}_1 \ := \ \left\{(\zt,\phi)\in \cplx\times\tor: {\rm Ker}(\mat(\zt,\phi))\supseteq 
			\{(V_1,\zt V_2)\in \CC:aV_1+b\zt V_2=0\}\right\},
\]
which is clearly a real-analytic subvariety of $\cplx\times\tor$. Assume that
$\mathcal{S}_1\bigcap(D(0;r)^*\times\tor)\neq \varnothing \ \forall r>0$. Here, and later in
this proof, $D(0;r)^*$ shall denote the punctured disc $D(0;r)\setminus\{0\}\subset \cplx$. Then,
there exists a sequence $\{(\zt_n,\phi_n)\}_{n\in \nat}\subset \cplx\times\tor$ such that
$\zt_n\lrarw 0$ as $n\to +\infty$ and
\[
 \mat(\zt_n,\phi_n) \begin{pmatrix}
			\ -b \ \\
			{} \\
			\ a \
			\end{pmatrix} \ = \ 0 \; \; \forall n\in \nat.
\]
This implies, in view of \eqref{E:leviformP}, that 
\[
 \levi{P}(\zt_n e^{i\phi_n}, e^{i\phi_n}; (-b\zt_n,a)) = 0 \; \; \forall n\in \nat.
\]
But this means that $P$ is {\em not} strictly plurisubharmonic on the set of
points\linebreak
$\{(\zt_n e^{i\phi_n}, e^{i\phi_n}):n\in \nat\}$, which has non-empty intersection 
with $\left(\scone{0}{\eps}\setminus L\right)$. This contradicts our hypothesis, whence our
assumption on $\mathcal{S}_1$ is false. Therefore, $\exists r_1\in (0,\eps)$ such that
$\mathcal{S}^{r_1}_1:=\mathcal{S}_1\bigcap(D(0;r_1)^*\times\tor)=\varnothing$.}
\smallskip

Let us define the following auxiliary objects:
\begin{align}
 \widetilde{\tail}(\zt,\phi;V) \ &:= \ |\zt|^2\langle \mat(\zt,\phi)V,V\rangle, \notag \\
 \mathfrak{u} \ &:= (-b,a)/\|(-b,a)\|. \notag
\end{align}
From the following facts:
\begin{itemize}
 \item $P$ is strictly plurisubharmonic on $\left(\scone{0}{\eps}\setminus L\right)$;
 \item $\mathcal{S}^{r_1}_1=\varnothing$, with $r_1<\eps$;
 \item $\levi{Q_\mu}((\zt\bcdot,\bcdot);(-b\zt,a))\equiv 0$;
\end{itemize}
we conclude that $\tail(\zt,\phi;(-b\zt,a))>0$ on $D(0;r_1)^*\times\tor$. Owing to
the real-analyticity of $\tail$, we can find an exponent $\eta : \mu+1\leq \eta\leq 2k$,
and a constant $\sigma_1>0$ such that:
\begin{itemize}
 \item[$(i)$] $\tail(\zt,\phi; (-b\zt,a)) \geq \sigma_1|\zt|^{\eta} \; \forall 
	(\zt,\phi)\in D(0;3r_1/4)\times\tor$.
 \item[$(ii)$] For each $\phi\in\tor$
 \[
  \liminf_{\zt\to 0}\frac{\tail(\zt,\phi; (-b\zt,a))}{|\zt|^{\eta}}>\sigma_1/2.
 \]
\end{itemize}
In view of the identity
\[
 \frac{\tail(\zt,\phi; (-b\zt,a))}{a^2+b^2} \ = \ \widetilde{\tail}(\zt,\phi;\mathfrak{u}),
\]
the bounds in $(i)$ and $(ii)$ above tell us that there is a 
$D(0;3r_1/4)\times\tor\times S^3$-open neighbourhood, say $\mathcal{U}$,
of $D(0;3r_1/4)\times\tor\times\{\mathfrak{u}e^{it}:t\in\RR\}$
such that
\[
 \widetilde{\tail}(\zt,\phi;V) \ \geq \ \frac{\sigma_1}{4(a^2+b^2)}|\zt|^{\eta}
	\; \; \forall (\zt,\phi,V)\in \mathcal{U}.
\]
Here, and later in this proof, $S^3$ shall denote the unit (Euclidean) sphere in $\CC$.
From the last inequality, it follows that there exists a small $S^3$-open neighbourhood
$\mathcal{W}$ of $\{\mathfrak{u}e^{it}:t\in\RR\}$ such that
\begin{equation}\label{E:tailPos11}
 \widetilde{\tail}(\zt,\phi;V) \ \geq \ \frac{\sigma_1}{4(a^2+b^2)}|\zt|^{\eta} \; \;
	\forall (\zt,\phi,V)\in \overline{D(0;r_1/2)}\times\tor\times\mathcal{W}.
\end{equation}
And finally, exploiting the relationship between $\widetilde{\tail}$ and $\tail$,
we infer from \eqref{E:tailPos11} that there exists a small constant $\beta>0$ --- which 
depends upon the cone determined by $\mathcal{W}$ --- such that
\begin{align}
 \text{\em If} \ |aV_1+b\zt V_2| \ &\leq \ \beta|a\zt V_2-bV_1| \ \text{\em then} \notag \\
 \tail(\zt,\phi;V) \ &\geq \ 
	\frac{\sigma_1}{4(a^2+b^2)}|\zt|^{\eta}\left(\left|\frac{V_1}{\zt}\right|^2+|V_2|^2\right),
	\; \; (\zt,\phi)\in (\overline{D(0;r_1/2)}\setminus\{0\})\times\tor.\notag
\end{align}
From this we can infer that, shrinking $\beta>0$ if necessary , if we define the closed cones
\[
 \vcone{\beta}(\zt) \ := \ \{V\in \CC: ((b/a)-\beta)|\zt V_2|\leq |V_1|\leq ((b/a)+\beta)|\zt V_2|\}
\]
(understanding $\beta$ to be smaller than $b/a$), then there are small constants 
$c_1, R_1>0$ such that
\begin{align}
 \tail(\zt,\phi;V) \ \geq \ c_1|\zt|^{\eta-2}&(|V_1|^2+|\zt V_2|^2) \notag \\
        \forall &(\zt,\phi,V)\in 
	\bigcup_{\widetilde{\zt}\in\overline{D(0;R_1)}}
	\{\widetilde{\zt}\}\times\tor\times\vcone{\beta}(\widetilde{\zt}). 
	\label{E:tailPos12}
\end{align}
\smallskip

\noindent{{\bf Step~2.} {\em The structure of the set 
$\{(\zt,w):w\zt\neq 0 \ \text{and} \ U_{\xi\xibar}(\zt^aw^{a+b})=0\}$ and definition of 
$H$ when this set is non-empty}}

\noindent{Let $\{\tht_1,\dots,\tht_M\}$ be the set associated to $\laplc U=4U_{\xi\xibar}$
as defined in Result~\ref{R:subhBump}. The arguments in this part of our proof {\em are
predicated on the assumption that $\{\tht_1,\dots,\tht_M\}\neq \varnothing$.} We will
consider what to do when $\{\tht_1,\dots,\tht_M\}= \varnothing$ in Step~3. We compute:
\begin{align}
 \mathfrak{S}_1 \ &:= \ \left\{(\zt,w)\in \CC: w\zt\neq 0 \ \text{and} \ 
					U_{\xi\xibar}(\zt^a w^{a+b})=0\right\} \notag \\
 &= \bigcup_{k=1}^M\bigcup_{r>0}\left\{(\zt,w)\in \CC: \zt^a w^{a+b}=re^{i\tht_k}\right\}.
	\label{E:radial}
\end{align}
We can apply a simple computation to \eqref{E:radial} to get:
\begin{equation}\label{E:rank0}
 \mathfrak{S}_1 \ = \ \bigcup_{\alpha\in \RR} \; \bigcup_{k=1}^M \; \bigcup_{l=1}^{a+b}
                        \left\{(re^{i\alpha},se^{i\phi}): r,s>0, \;
                                \phi=\frac{\tht_k-a\alpha+2\pi l}{a+b}\right\}.
\end{equation}
In other words, for each fixed $\zt\in \cplx\setminus\{0\}$, $\mathfrak{S}_1\bigcap(\{\zt\}\times\cplx)$
is a collection of radial segments. We now apply Result~\ref{R:subhBump} to $U$ and
let $h$ be the $2\pi$-periodic function whose existence is asserted by Result~\ref{R:subhBump}.
Let $m\in \zahl_+$ be such that $U$ is homogeneous of degree $2m$.
Define
\[
 F(re^{i\tht}) \ := \ r^{2m}h(\tht) \; \; \forall re^{i\tht}\in \cplx.
\] 
Now set
\begin{multline}
 A_k \ := \ \frac{\sup\{t>0: \laplc(U-\delta F)(e^{i\tht})>\laplc U(e^{i\tht}) \; \forall
			\tht\in (\tht_k-t,\tht_k+t)\}}{2}, \\ 
			k=1,\dots M, \;\; 0< \delta\leq 1. \notag
\end{multline}
The whole point of Part~$(b)$ of Result~\ref{R:subhBump} is that $A_k>0 \ \forall k\leq M$.
It is obvious that $A_k$ is independent of $\delta$ because any lowering of 
$\laplc(U-\delta F)(e^{i\tht})$ with respect to $\laplc U(e^{i\tht})$ is determined {\em only by
the cut-off functions used to construct $h$}. 
\smallskip

Let us now define 
\begin{equation}\label{E:HBad}
  H(z_1,z_2) \ := \ F(z_1^a z_2^b) \; \; \forall (z_1,z_2)\in \CC.
\end{equation}
Let us also set $A:=\min_{1\leq k\leq M}A_k$.
Refer to Result~\ref{R:subhBump} for a definition of the sets $\mathcal{S}_k(A)$, which
are just closed sectors in $\cplx$. If we now define the set $\omega\subset \CC$ by
\[
 \omega \ := \  \bigcup_{\alpha\in \RR} \; \bigcup_{k=1}^M \; \bigcup_{l=1}^{a+b}
		\left\{(re^{i\alpha},se^{i\phi}): r,s\geq 0, \;
			\left|\phi-\frac{\tht_k-a\alpha+2\pi l}{a+b}\right|\leq \frac{A}{a+b}\right\},
\]
then, comparing this with \eqref{E:rank0} and the definition of $A$, we infer:
\begin{equation}\label{E:sectors}
 (\zt,w)\in \omega \; \iff \ \zt^a w^{a+b}\in \mathcal{S}_k(A) \; \text{for some $k=1,\dots,M$.}
\end{equation}
Then, from the definition of $A$, we conclude that
\begin{align}
 \levi(Q_\mu-\delta H)(\zt w,w;V) \ &= 
		\ (U-\delta F)_{\xi\xibar}(\zt^a w^{a+b})J(\zt,w)|aV_1+b\zt V_2|^2 \notag \\
	&\geq \ \levi Q_\mu(\zt w,w;V) 
	\;\; \forall (\zt,w,V)\in \omega\times\CC, \label{E:headPos}
\end{align}
for each $\delta: 0< \delta\leq 1$.}
\smallskip

\noindent{{\bf Step~3.} {\em Completing the proof when the set 
$\{(\zt,w):w\zt\neq 0 \ \text{and} \ U_{\xi\xibar}(\zt^aw^{a+b})=0\}$
is non-empty}}

\noindent{In view of Part~$(b)$ of Result~\ref{R:subhBump}, the definition of $H$ and the 
estimate \eqref{E:tailPos12}, we have
\begin{multline}\label{E:leadBump1}
 \levi(P-\delta H)(\zt w,w;V) \ \geq \  
		c_1|w|^{2(k-1)}|\zt|^{\eta-2}(|V_1|^2+|\zt V_2|^2) \\
	\forall(\zt,w,V)\in \bigcup_{\widetilde{\zt}\in\overline{D(0;R_1)}}
				\{\widetilde{\zt}\}\times\tor\times\vcone{\beta}(\widetilde{\zt}),
\end{multline}
for each $\delta: 0<\delta\leq 1$. 
Next, from \eqref{E:headPos} and \eqref{E:leviformP}, we conclude that
\begin{equation}\label{E:leadBump2}
 \levi(P-\delta H)(\zt w,w;V) \ \geq \levi{P}(\zt w,w;V) \ \geq \ 0 \; \;
					\forall (\zt,w,V)\in \omega\times\CC,
\end{equation}
for each $\delta: 0<\delta\leq 1$.}
\smallskip

Let us stay with the case when $(\laplc U)^{-1}\{0\}\varsupsetneq \{0\}$ for the moment
and complete our argument for this case. It now remains to account for the positivity of the Levi-form 
of $(P-\delta H)$ along
vectors that remain unaccounted for by the inequality \eqref{E:leadBump1}. We will
have to treat this under separate two cases. In this part of the proof, we shall use the
notation $J_l(\zt,w)$ to denote the product $|\zt|^{\mu-l}|w|^{2(k-1)}, \ l=0,\dots,\mu$.
\medskip

\noindent{{\em Case $(i)$} {\em $V\in \CC: \ |V_1|\leq ((b/a)-\beta)|\zt V_2|$ for some $\zt\in \cplx$}}

\noindent{The above condition implies
\begin{align}
 |aV_1+b\zt V_2| \ &\geq \ b|\zt V_2|-a|V_1| \ \geq \ a\beta|\zt V_2|, \notag \\
 |V_1|^2+|\zt V_2|^2 \ &\leq \left((b/a)^2+1\right)|\zt V_2|^2. \notag
\end{align}
It sufficies to study $\levi(P-\delta H)$ for 
\begin{align}
 (\zt,w,V) \ \in \ & \left\{(\zt,w,V)\in \omega^{{\sf C}}\times\CC: |\zt|<1/2, \ 
				|V_1|\leq ((b/a)-\beta)|\zt V_2|\right\} \notag \\ 
	&:= \ W_1(s=1/2). \notag
\end{align}
We appeal to the observation \eqref{E:sectors}, which, in conjunction with Part~(c) of
Result~\ref{R:subhBump}, tells us that there exist constants $C,M>0$ such that (exploiting
the two inequalities above) we may estimate:
\begin{align}
 \levi(P-\delta H)(\zt w,w;V) \ &\geq \ C(a\beta)^2J_0(\zt,w)|V_2|^2
					-MJ_1(\zt,w)(|V_1|^2+|\zt V_2|^2) \notag \\
		&\geq 
	\ J_0(\zt,w)\left[C(a\beta)^2-M|\zt|\left((b/a)^2+1\right)\right]|V_2|^2. \notag \\
		& \qquad \forall (\zt,w,V) \ \in \ W_1(1/2), \notag
\end{align}
for each $\delta: 0<\delta\leq 1$. Thus, there exists a positive constant $R_2\ll 1$ such that 
\begin{align}
 \levi(P-\delta H)(\zt w,w;V) \ &\geq \frac{C(a\beta)^2}{2}J_0(\zt,w)|V_2|^2 \notag \\
	&\geq \ c_2J_2(\zt,w)(|V_1|^2+|\zt V_2|^2) \label{E:leadBump3} \\
	& \qquad \forall (\zt,w,V) \ \in \ W_1(R_2), \notag
\end{align}
for each $\delta: 0<\delta\leq 1$.}
\smallskip

\noindent{{\em Case $(ii)$} {\em $V\in \CC: \ |V_1|\geq ((b/a)+\beta)|\zt V_2|$ for some $\zt\in \cplx$}}

\noindent{The inequalities analogous to the ones used in Case~$(i)$ are:   
\begin{align}  
 |aV_1+b\zt V_2| \ &\geq \ a|V_1|-b|\zt V_2| \ \geq \   
                a\beta\left((b/a)+\beta\right)^{-1}|V_1|, \notag \\
 |V_1|^2+|\zt V_2|^2 \ &\leq \left[\left((b/a)+\beta\right)^{-2}+1\right]|V_1|^2. \notag
\end{align}
It is now evident that the previous argument goes through {\em mutatis mutandis} to yield
constants $c_3,R_3>0$ such that
\begin{align}\label{E:leadBump5}
 \levi(P-\delta H)(\zt w,w;V) \
        &\geq \ c_3J_2(\zt,w)(|V_1|^2+|\zt V_2|^2) \\
        & \qquad \forall (\zt,w,V) \ \in \ W_2(R_3), \notag
\end{align}
for each $\delta: 0<\delta\leq 1$, where
\[
 W_2(s) \ := \ \left\{(\zt,w,V)\in \omega^{{\sf C}}\times\CC: |\zt|<s, \
                        |V_1|\geq ((b/a)+\beta)|\zt V_2|\right\}.
\]}

From the definition of $H$ given by \eqref{E:HBad} (recall that $(\laplc U)^{-1}\{0\}\varsupsetneq \{0\}$),
Part~$(a)$ is evident. Next; if we define
\[
 \sigma \ := \ \min\{\eps, R_1, R_2, R_3\}
\]
and make the following substitutions
\[
 z_1 \ = \ \zt w, \qquad\quad z_2 \ = \ w,
\]
then, the statements \eqref{E:leadBump1}, \eqref{E:leadBump2}, \eqref{E:leadBump3}
and \eqref{E:leadBump5} clearly establish Part~$(b)$
\smallskip

\noindent{{\bf Step~$\boldsymbol{3^\prime}$.} {\em Completing the proof when the 
set $\{(\zt,w):w\zt\neq 0 \ \text{and} \ U_{\xi\xibar}(\zt^aw^{a+b})=0\}$ is empty}}

\noindent{Very similar ideas will work when $(\laplc U)^{-1}\{0\}=\{0\}$. In this situation, 
Step~2 is irrelevant. Instead, we begin by defining
$\gamma:=\inf_{\tht\in[0,2\pi]}U_{\xi\xibar}(e^{i\tht})$. Note that $\gamma>0$.
In this case, we define (recall that $U$ is homogeneous of degree $2m$)
\begin{equation}\label{E:HGood}
  H(z_1,z_2) \ := \ \frac{\gamma}{2m^2}|z_1^a z_2^b|^{2m} \; \; \forall (z_1,z_2)\in \CC.
\end{equation}
In this situation, we can repeat the entire argument given in Step~3 --- relying, once again, on the 
two inequalities that follow the headings ``Case~$(i)$'' and ``Case~$(ii)$'' --- making the
following replacements:
\begin{itemize}
 \item Replace the definition of $W_1(s)$ given in Case~$(i)$ by
 \[
   W_1(s) \ := \ \left\{(\zt,w,V)\in \CC\times\CC: |\zt|<s, \ 
			|V_1|\leq ((b/a)-\beta)|\zt V_2|\right\};
 \]
 \item  Replace the definition of $W_2(s)$ given in Case~$(ii)$ by
 \[
  W_2(s) \ := \ \left\{(\zt,w,V)\in \CC\times\CC: |\zt|<s, \
                        |V_1|\geq ((b/a)+\beta)|\zt V_2|\right\};
 \]
 \item Replace the constant $C$, wherever it occurs in Step~3, by the constant $\gamma/2$.
\end{itemize}
It is easy to see that this exercise leads us to the following conclusion:
\begin{multline}
 \text{\em When $(\laplc U)^{-1}\{0\}=\{0\}$, analogues of the estimates \eqref{E:leadBump3}
 and \eqref{E:leadBump5}} \\
 \text{\em are achieved with $H$ as
 redefined in \eqref{E:HGood} $($with the newly-defined $W_j(s)$, $j=1,2$$)$,}
\end{multline}
from which the desired result follows for the case $(\laplc U)^{-1}\{0\}=\{0\}$.}
\end{proof}

\section{The key proposition}\label{S:homoBump}

We are now in a position to prove the key technical proposition of this work. We ought
to indicate that the work in this section represents a continuation of the work done in
\cite{bharaliStensones:psb09}. While the philosophy of the proof of Proposition~\ref{P:homoBump}
below is not very different from the proof of \cite[Theorem~1]{bharaliStensones:psb09}, the
{\em estimates needed for Theorem~\ref{T:main} affect the details of the proof below}. For
instance: the proof of Proposition~\ref{P:homoBump} would not work without 
Proposition~\ref{P:leadBump}. The similarity with \cite[Theorem~1]{bharaliStensones:psb09} is
further visible in our use of the other crucial idea needed in our proof. This result is derived 
from \cite{Noell:pfpd93}.
\smallskip

\begin{result}[Version of Prop.~4.1 in \cite{Noell:pfpd93}]\label{R:noell}
Let $P$ be a plurisubharmonic polynomial on $\Cn$ that is homogeneous of degree $2k$. Let   
$\omega_0$ be a connected component of $\LeviDeg{P}\setminus\{0\}$ having the following
two properties:
\begin{itemize}  
\item[$a)$] There exist closed cones $\clcone_1$ and $\clcone_2$ such that
\[
\omega_0 \subset \ {\rm int}(\clcone_1) \subset \ \clcone_1\setminus\{0\}
\varsubsetneq \ {\rm int}(\clcone_2),
\]
and such that $\clcone_2\bigcap(\LeviDeg{P}\setminus\overline{\omega_0})=\varnothing$.
\item[$b)$] $\omega_0$ does not contain any complex-analytic subvarieties of positive dimension
along which $P$ is harmonic.
\end{itemize} 
Then, there exist a smooth function $H\geq 0$ that is homogeneous of degree $2k$ and constants
$C,\eps_0>0$, which depend only on $P$, such that  ${\rm int}(\clcone_2)=\{H>0\}$ and such that
for each $\eps:0<\eps\leq\eps_0$, $\levi{(P-\eps H)}(z;v)\geq C\eps\|z\|^{2(k-1)} \   
\|v\|^2 \ \forall (z,v)\in\clcone_2\times\Cn$.
\end{result}

\begin{remark}
The above result is not stated in precisely these words in
\cite[Proposition~4.1]{Noell:pfpd93}. The proof of the latter was derived from
a construction pioneered by Diederich and Fornaess in \cite{diederichFornaess:pd:bspef77}.
A careful comparison of the proof of \cite[Proposition~4.1]{Noell:pfpd93} with the
Diederich-Fornaess construction ---  keeping in mind the assumption of homogeneity in 
Result~\ref{R:noell} --- easily reveals that the assumption $(a)$ in Result~\ref{R:noell} 
is enough to obtain the above ``localised''
version of \cite[Proposition~4.1]{Noell:pfpd93}.
\end{remark}
\smallskip

\begin{prop}\label{P:homoBump} Let $\mathcal{U}$ be a neighbourhood of $0\in \CCC$
and let $\mathcal{V}$ be such that $\{0\}\times\mathcal{V}=\mathcal{U}\bigcap(\{0\}\times\CC)$.
Let $R\in \smoo^\infty(\mathcal{V})\bigcap\psh(\mathcal{V})$ and be such that the domain
$\OM_R:=\{(w,z)\in \mathcal{U}:\er{w}+R(z)<0\}$ is of finite type. Let
\[
 \sum_{n=2k}^\infty P_{n,R}(z_1,z_2)
\]
denote the Taylor expansion of $R$ around $z=0$, where $k\geq 2$, and
each $P_{n,R}$ is the sum of all monomials having total degree $n$. Set $P:=P_{2k,R}$. Assume that
$P$ has no pluriharmonic terms and has the properties stated in Definition~\ref{D:almost-h-ext}.
If $\exepc(P)\neq \varnothing$, then:
\begin{enumerate}
 \item[1)] $\exepc(P)$ is a finite collection of complex lines $L_1,\dots L_N$ passing through
 the origin.

 \item[2)] For any $x\in L_j$, set $[j,x]^\perp:=(x+L_j^\perp), \ j=1,\dots, N$, where the orthogonal
 complement is taken with respect to the standard inner product on $\CC$. Assume that for
 each $x\in L_j\setminus\{0\}$
 \begin{equation}\label{E:simp}
  Ord\left(\left. P\right|_{[j,x]^\perp}-P(x),x\right) \ \leq \
        Ord\left(\left. R\right|_{[j,x]^\perp}-R(x),x\right),
 \end{equation}
 and let $\mu_j$ denote the generic value on $L_j$ of the left-hand side of \eqref{E:simp},
 $j=1,\dots, N$. Then, there is an {\em algebraic} change of coordinate centered at $0\in \CCC$
 and a positive integer $D$, chosen suitably so as to give us $(a)$--$(d)$ below, such that if
 $(W,Z_1,Z_2)$ denotes the new coordinates, then $Z=z$; and if\linebreak 
 $\OM_R=\{(W,z): \er{W}+\rho(z)<0\}$ is the representation of $\OM_R$ relative to
 the new coordinates, then the associated $P_{n,\rho}$ contain no pluriharmonic terms
 for $n=2k,\dots,D$, and $P_{2k,R}=P_{2k,\rho}=P$. 
 Furthermore, there exist constants $C, \delta_0, r_0>0$; a non-negative function
 $H\in \smoo^\infty(\CC\setminus\{0\})\bigcap \smoo^2(\CC)$ that is homogeneous of degree
 $2k$; closed cones $\clcone^1_j$ and $\clcone^2_j$ satisfying
 \[
  \clcone^1_j\setminus\{0\} \varsubsetneq {\rm int}(\clcone^2_j), \;\; j=1,\dots,N;
 \]
 and, for each $\delta\in (0,\delta_0)$, a constant $r_\delta>0$; subharmonic functions
 $U_{j,\delta}\in \smoo^\infty(\cplx\setminus\{0\})\bigcap \smoo^2(\cplx)$ that are homogeneous
 of degree $2M_j, \ j=1,\dots, N$, and are strictly subharmonic away from $0\in \cplx$; 
 and a function $G_\delta\in \smoo^\infty(\CC)\bigcap\psh(\CC)$ such that:
 \begin{itemize}
  \item[$a)$] $G_\delta$ is strictly plurisubharmonic on $\CC\setminus\{0\}$.
  \item[$b)$] For each $j\leq N$, $G_\delta(z) = 
  (P-\delta H)(z)+U_{j,\delta}(z_{k(j)}) \;\; \forall (z_1,z_2)\in
        \clcone^1_j\bigcap \nball{2}(0;r_0)$, where $k(j)=1 \ \text{or} \ 2$ depending on $j=1,\dots, N$.
  \item[$c)$] $G_\delta(z_1,z_2)=(P-\delta H)(z_1,z_2) \;\;
        \forall (z_1,z_2)\in \left(\CC\setminus \cup_{j=1}^N\clcone^2_j\right)\bigcap \nball{2}(0;r_0)$.
  \item[$d)$] For each $j\leq N$ such that the complex line $L_j$ has the form $L_j=\{z:z_1=\xi_j z_2\}$,
  $\xi_j\in \cplx$,
  \begin{align}
   (G_\delta-\rho)(z)\leq -\delta C(|z_1-\xi_j z_2|^{\mu_j}|z_2|^{2k-\mu_j}+&|z_2|^{2M_j}) \notag \\
		&\forall (z_1,z_2)\in \clcone^1_j\bigcap \nball{2}(0;r_\delta). \notag
  \end{align}
  If $L_j=\{z:z_2=0\}$, then the above inequality holds with the positions of $z_1$ and $z_2$ swapped
  and $\xi_j=0$.
  \item[$e)$] $(G_\delta-\rho)(z)\leq -\delta C\|z\|^{2k} \;\;
        \forall (z_1,z_2)\in \left(\CC\setminus \cup_{j=1}^N\clcone^1_j\right)\bigcap 
	\nball{2}(0;r_\delta)$.
 \end{itemize}
\end{enumerate}
\end{prop}
\begin{proof}
Note that Part~(1) follows simply from Proposition~\ref{R:harmlines}.
Let $\clcone$ be the closed cone with the properties given in Definition~\ref{D:almost-h-ext}
whose existence is guaranteed by hypothesis. By assumption, 
we can find a slightly larger cone $\clBcone$ such that
\[
\LeviDeg{P}\setminus(\bigcup_{j=1}^N L_j) \subset \ {\rm int}(\clcone)
\subset \ \clcone\setminus\{0\} \varsubsetneq \ {\rm int}(\clBcone),
\]
and such that $\clBcone\bigcap(\bigcup_{j=1}^N L_j)=\{0\}$. 
\smallskip

\noindent{{\bf Step~1.} {\em Constructing $G_\delta$}}

\noindent{In view of the assumption that $P$ has no pluriharmonic terms,
$\left. P\right|_{L_j}\equiv 0$. {\em{\bf{\em Please note:}} in the interests of brevity,
our arguments below will be framed for complex lines in $\exepc(P)$ of the form}
\[
 L_j \ = \ \{z:z_1=\xi_j z_2\}, \; \; \xi_j\in \cplx,
\]
{\em only. If $\exepc(P)\ni \{z:z_2=0\}=:L_{j^0}$, then a separate argument will be provided
for a statement relating to $L_{j^0}$ only if that statement does not follow {\em mutatis
mutandis} from the arguments given.} In particular: for the purposes of Step~1, we may
assume that $L_j=\{z:z_1=\xi_j z_2\}, \ j=1,\dots,N$.
The hypotheses of Proposition~\ref{P:leadBump} are satisfied by
\[
 \jayP(z_1,z_2) \ := \ P(z_1+\xi_j z_2,z_2).
\]
Hence, we can find:
\begin{itemize}
 \item a constant $c_1>0$ that depends only on $P$;
 \item constants $\sigma_j>0, \ j=1,\dots,N$, that depend only on $P$ and $j$;
 \item functions $H_j\in\smoo^\infty(\CC\setminus\{0\})\bigcap\smoo^2(\CC)$
 that are homogeneous of degree $2k$; and
 \item a positive decreasing function $B:(0,1)\lrarw \RR$;
\end{itemize}
such that
\begin{align}
 \levi{(P-\delta H_j)}(z;(V_1,V_2)) \ \geq \ &B(t)\|z\|^{2(k-1)}\|V\|^2 \notag \\
        &\forall z: t\sigma_j|z_2|<|z_1-\xi_jz_2|<\sigma_j|z_2|, \; \; \forall V\in\CC,
	\label{E:leviBS} \\
 H_j(z_1,z_2) \ > \ &c_1|z_1-\xi_jz_2|^{\mu_j}|z_2|^{2k-\mu_j} \;\;
 	\forall z\in \scone{\xi_j}{\sigma_j}\setminus L_j, \label{E:sizej}
\end{align}
for $t\in (0,1)$, $j=1,\dots,N$, and for $\delta\in (0,1]$. Recall that $\mu_j\in \zahl_+$ are
as determined by the condition \eqref{E:simp}.}
\smallskip

Note, furthermore, that, by the finite-type assumption on $\OM_R$, if we define
\[
 M_j \ := \ \frac{\inf\{n\in \zahl_+:\left. P_{n,R}\right|_{L_j} \ 
 \text{\em is non-harmonic on $L_j$}\}}{2},
\]
then $M_j<\infty \ \forall j=1,\dots,N$. Let us now write $D:=\max_{j\leq N}2M_j$. 
It is a well-known trick that we can make a global holomorphic change of coordinates and 
work with new coordinates $(W,Z_1,Z_2)$ having the form
\begin{align}
 W \ &= \ w-(\text{\em a holomorphic polynomial in $z_1$ and $z_2$ of degree\,$\leq D$}), \notag \\
 Z \ &= \ z, \label{E:aboutAss}
\end{align}
such that if the representation of $\OM_R$ with respect to these coordinates is
\[
 \OM_R = \{(W,Z): \er{W}+\rho(Z)<0\},
\]
then the Taylor expansion of $\rho$ around $Z=0$ contains no pluriharmonic
terms of degree\,$\leq D$, and $P_{2k,R}=P_{2k,\rho}$. To clarify: the holomorphic polynomial 
that is a part of the definition of $(W,Z)$ is the sum of all monomials in $z_1$ and $z_2$ whose real 
and imaginary parts are pluriharmonic terms in $P_{n,R}, \ n=2k+1,\dots,D$. Let us call this
polynomial $\Theta$. Now, owing to the property \eqref{E:aboutAss} of $(W,Z_1,Z_2)$:
\begin{itemize}
 \item The condition \eqref{E:simp} remains unchanged when $R$ is replaced by $\rho$. To 
 understand this, first note that the left-hand side of \eqref{E:simp} remains unchanged.
 Now, suppose $L_j={\rm span}_{\cplx}\{(B_1,B_2)\}$. If $\Odr_{x,j}\in[1,\infty]$ denotes
the degree of the first non-zero monomial in $\tau$ and $\overline{\tau}$ of the Taylor
expansion in $(\tau,\overline{\tau})$ (around $\tau=0$) of 
$(R(x+\tau(B_2,-B_1))-R(x))$, then \eqref{E:simp} states that
\[
 \Odr_{x,j} \ \geq \ Ord\left(\left. P\right|_{[j,x]^\perp}-P(x),x\right).
\]
Now, as the monomials of the polynomial $(\Theta(x+\tau(B_2,-B_1))-\Theta(x))$ occur in the
aforementioned Taylor expansion, the degree of the first non-zero monomial in $\tau$ and 
$\overline{\tau}$ of the Taylor expansion in $(\tau,\overline{\tau})$ of
$(\rho(x+\tau(B_2,-B_1))-\rho(x))$ cannot be less than $\Odr_{x,j}$. Hence our claim.
 \item None of the assertions made thus far in this proof are affected.
\end{itemize}
From this point, we shall work with the defining function $(\er{W}+\rho(z))$, and we shall
simply write $P_{n,\rho}=:P_n, \ n\in \zahl_+$. It follows from an argument very similar to 
the one used in the proof of Lemma~\ref{L:loplush} that the functions  
\[
 u_j(s) \ := \ P_{2M_j}(\xi_j s,s), \; \; s\in \cplx,
\]
$j=1,\dots,N$, are subharmonic, non-harmonic functions. In particular, therefore, the $M_j$'s
defined above are integers. We invoke Result~\ref{R:subhBump} to obtain:
\begin{itemize}
 \item constants $c_2, B_2>0$ that depend only on $\rho$; and
 \item functions $h_j\in\smoo^\infty(\cplx\setminus\{0\})\bigcap\smoo^2(\cplx)$
 that are homogeneous of degree $2M_j, j=1,\dots,N$;
\end{itemize}
such that
\begin{align}
 h_j(s) \ &\geq \ c_2|s|^{2M_j} \; \; \forall s\in \cplx, \label{E:sizej2} \\
 \Mixdrv{s}{s}(u_j-\delta h_j)(s) \ &\geq \ \delta B_2|s|^{2(M_j-1)} \; \; \forall s\in \cplx, \;
	\forall\delta: 0<\delta\leq 1, \label{E:subhPos}
\end{align}
$j=1,\dots,N$. Lastly, in view of
Result~\ref{R:noell}, we can find a smooth function $H_0\geq 0$ that is homogeneous
of degree $2k$, and constants $B_3,\eps_0>0$ such that
\begin{align}
 \{z:H_0>0\} \ &= \ {\rm int}(\clBcone), \notag \\
 \levi(P-\delta H_0)&(z;V) \notag \\
 &\geq \ \delta B_3\|z\|^{2(k-1)} \ \|V\|^2 \quad
 \forall (z,v)\in\clBcone\times\CC, \ \text{and} \ \forall\delta\in(0,\eps_0).\label{E:leviNoell}
\end{align}

Let $\al>0$ be so small that
\begin{align}
 2\al \ &\leq \ \sigma_j, \; \;  j=1,\dots,N,\notag\\
 (\smcone{\xi_j}{2\al}\cap S^3)\bigcap(\clcone_*\cap S^3) \ &= \ \varnothing \quad\forall j\leq N,\notag\\
 (\smcone{\xi_j}{2\al}\cap S^3)\bigcap(\smcone{\xi_k}{2\al}\cap S^3) \ &= \ \varnothing
 \quad\text{if $j\neq k$}.
\end{align}  
Here, $S^3$ denotes the unit sphere in $\CC$. The parameter $\alpha$ will be used to define the 
cones $\clcone^1_j$ and $\clcone^2_j, \ j=1,\dots,N$. Let us now define
\[
 V_j \ := \ \smcone{\xi_j}{\al}\cap S^3, \quad\text{and}\quad
        W_j \ := \ \smcone{\xi_j}{2\al}\cap S^3.
\]
Let $\chi_j:S^3\lrarw [0,1]$ be a smooth cut-off function such that
$\chi_j|_{V_j}\equiv 1$ and ${\rm supp}(\chi_j)\subset W_j, \ j=1,\dots,N$. Let us
define $\Psi_j(z):=\chi_j(z/\|z\|) \ \forall z\in\CC\setminus\{0\}$. Finally, if 
$\Phi:\CC\lrarw \RR$ is a function that is homogeneous of degree $d>0$, then we shall
abuse notation somewhat and use the 
expression $\Psi_j(z)\Phi(z)$ as having the following meaning:
\[
 \Psi_j(z)\Phi(z) \ := \ \begin{cases}
                        \Psi_j(z)\Phi(z), &\text{if $z\neq 0$,} \\
                        0, &\text{if $z=0$.}
                        \end{cases}
\]
Note that $\Psi_j\Phi$ has the same regularity as $\Phi$ 
and is homogeneous of degree $d$. Before we define $G_\delta$, let us define a preliminary
function $\gee_\delta$:
\begin{equation}\label{E:gee}
 \gee_\delta(z) \
 := \ P(z)-\delta H_0(z)+\sum_{j=1}^N\left[u_j(z_2)-\delta\left(H_j(z)+h_j(z_2)\right)\right]\Psi_j(z)
	\; \; \forall z\in \CC.
\end{equation}
{\em{\bf{\em Note that}} if $\exepc(P)\ni \{z:z_2=0\}=:L_{j^0}$, then the sum on the right-hand side
above would contain the summand 
$\left[u_{j^0}(z_1)-\delta\left(H_{j^0}(z)+h_{j^0}(z_1)\right)\right]\Psi_{j^0}(z)$}. 
Also note that, as $\Psi_j\equiv 1$ on $\smcone{\xi_j}{\alpha}\bigcap S^3$, 
Proposition~\ref{P:leadBump} and \eqref{E:subhPos} imply that $\gee_\delta$ is {\em strictly}
plurisubharmonic on $\smcone{\xi_j}{\al}\setminus\{0\}, \ j=1,\dots,N$. But, since strict 
plurisubharmonicity is an open condition, we infer from continuity and homogeneity that
\begin{itemize}
 \item[$(i)$] $\exists \eps\ll 1$ such that
 $\gee_{\delta}$ is strictly plurisubharmonic on $\scone{\xi_j}{\al+\eps}$;
\end{itemize}
for each $j=1,\dots,N$, and for each $\delta:0<\delta\leq 1/2$.
In view of $(i)$ above we need to examine the Levi-form of $\gee_\delta$ on
$\scone{\xi_j}{2\alpha}\setminus\smcone{\xi_j}{\al+\eps}$. 
By our definition of the $\Psi_j$'s, we can find a $\beta>0$ such that
\[
(1-\Psi_j)(z) \ \geq \ \beta \quad \forall z\in
                \scone{\xi_j}{2\al}\setminus\smcone{\xi_j}{\al+\eps}, \; \; j=1,\dots,N.
\]
Furthermore, an application of \eqref{E:leviBS}, and an appeal to the properties of 
$P$ stated in Definition~\ref{D:almost-h-ext},
respectively, imply that there exists a constant $\gamma>0$ such that
\[
 z\in (\scone{\xi_j}{2\al}\setminus\smcone{\xi_j}{\al+\eps}) 
 \Rightarrow 
 \begin{cases}
  \levi{(P+u_j-\delta(H_j+h_j))}(z;V)\!\!\negthickspace&\geq 0, \\
  \qquad\qquad\qquad\qquad\quad \levi{P}(z;V)\!\!\negthickspace&\geq \gamma\|z\|^{2(k-1)}\|V\|^2,
 \end{cases}
\]
for each $V\in \CC$ and each $j=1,\dots,N$, and for each $\delta:0<\delta\leq 1/2$. For 
each $j=1,\dots,N$, let
us write $F_j(z_1,z_2):=H_j(z_1,z_2)+h_j(z_2)$. From the last three inequalities,
we can estimate:
\begin{align}
\levi{\gee_{\delta}}(z;V) \ &\geq \ 
 (1-\Psi_j)(z)\levi{P}(z;V)+\left(u_j(z_2)-\delta F_j(z)\right)\levi{\Psi_j}(z;V) \notag \\
	&\qquad +2\er\left[\sum_{\mu,\nu\leq 2}
		\monodrv{\mu}\Psi_j(z)\monodrv{\overline{\nu}}\left(u_j-\delta F_j\right)(z)
                V_\mu\overline{V_\nu}\right] \notag\\
 &\geq \ \beta\gamma\|z\|^{2(k-1)}\|V\|^2 -\delta S_1(z;V)- S_2(z;V) \label{E:geeLeviEst} \\
 &\qquad\qquad \forall z\in\scone{\xi_j}{2\al}\setminus\smcone{\xi_j}{\al+\eps} \
                \text{\em and} \ \forall V\in\CC, \notag
\end{align}
where $S_1(z;V)$ and $S_2(z;V)$ are as follows. We first consider $S_1(z;V)$, in which case we can
find a large constant $K_1>0$ such that
\begin{multline}\label{E:S_1}
 S_1(z;V) \ := \ 2\sum_{\mu,\nu\leq 2}
                \left|\monodrv{\mu}\Psi_j(z)\monodrv{\overline{\nu}}H_j(z)   
                V_\mu\overline{V_\nu}\right| + H_j(z)\left|\levi{\Psi_j}(z;V)\right| \\
		\leq \ K_1\|z\|^{2(k-1)}\|V\|^2 \; \; \;
		\forall (z,V)\in(\scone{\xi_j}{2\al}\setminus\smcone{\xi_j}{\al+\eps})\times\CC,
\end{multline}
for each $j=1,\dots,N$. In a similar way, we can find a large constant $K_2>0$ such that
\begin{equation}\label{E:S_2}
 S_2(z;V) \ \leq \ K_2|z_2|^{2(M_j-1)}\|V\|^2 \; \; \; \forall
(z,V)\in(\scone{\xi_j}{2\al}\setminus\smcone{\xi_j}{\al+\eps})\times\CC, 
\end{equation}
for each $j=1,\dots,N$. Let $\delta^*>0$ be so small that 
$\delta K_1\leq\beta\gamma/4 \ \forall \delta\in (0,\delta^*]$. Then, we can find an $r_0>0$
such that, in view of \eqref{E:S_1} and \eqref{E:S_2}, the following holds:
\begin{align}
 \levi{\gee_{\delta}}(z;V) \ 
	&\geq (\beta\gamma-\delta K_1)\|z\|^{2(k-1)}-K_2|z_2|^{2(M_j-1)}\|V\|^2 \notag \\
			   \ &\geq \frac{\beta\gamma}{2}\|z\|^{2(k-1)}\|V\|^2 \label{E:geePos}
				\\
 &\qquad\qquad \forall z\in(\scone{\xi_j}{2\al}\setminus\smcone{\xi_j}{\al+\eps})\bigcap\nball{2}(0;2r_0), 
		\notag \\
 &\qquad\qquad\forall V\in\CC,  \; \text{\em and} \; \forall\delta: 0<\delta\leq \delta^*,\notag
\end{align}
for each $j=1,\dots,N$.
Let us now set
\begin{align}
 \widetilde{H} \ &:= \ H_0+\sum_{j=1}^N\Psi_jH_j, \notag\\
 \delta_0 \ &:= \ \min(\eps_0,\delta^*). \notag
\end{align}
So far, in view of $(i)$ above and the bound \eqref{E:geePos}, we have accomplished the 
following:
\begin{enumerate}
 \item[$(ii)$] $\gee_{\delta}$ is {\em strictly} plurisubharmonic on $\nball{2}(0;2r_0)\setminus\{0\}$
 $\forall\delta\in(0,\delta_0)$.
 \item[$(iii)$] $\{z:\widetilde{H}>0\}=
 {\rm int}(\clcone_*)\bigcup(\bigcup_{j=1}^N(\scone{\xi_j}{2\al}\setminus L_j))$.
\end{enumerate}
We carry out the following three steps to transform $\gee_\delta$ to $G_\delta$:
\begin{itemize}
 \item[A)] We make a small perturbation of $\widetilde{H}$ to obtain a function
 $H$ having the same regularity as $\widetilde{H}$ and homogeneous of degree $2k$
 such that $H^{-1}\{0\}=\bigcup_{j=1}^N L_j$.
 \item[B)] Ensure that the perturbation in (A) is so small that if we set
 \begin{equation}\label{E:Gamma}
  \widetilde{\Gamma}_\delta(z):=P(z)-\delta H(z)
	+\sum_{j=1}^N\left(u_j(z_2)-\delta h_j(z_2)\right)\Psi_j(z)
        \; \; \forall z\in \CC,
 \end{equation}
 then $\widetilde{\Gamma}_\delta$ is strictly plurisubharmonic on $\nball{2}(0;2r_0)\setminus\{0\}$
 $\forall\delta\in(0,\delta_0)$. {\em The process of achieving this has been described in
 the proof of \cite[Theorem~1]{bharaliStensones:psb09} and, hence, we shall not repeat this argument.}
 \item[C)] Since $\widetilde{\Gamma}_\delta$ is strictly plurisubharmonic on 
 $\nball{2}(0;2r_0)\setminus\{0\}$, it is a well-known fact that we can extend 
 $\widetilde{\Gamma}_\delta$ to a function $G_\delta$ on $\CC$ that is strictly plurisubharmonic
 on $\CC\setminus\{0\}$ such that
 \begin{equation}\label{E:agree}
  \left. \widetilde{\Gamma}_\delta\right|_{\nball{2}(0;r_0)} \ = \
  \left. G_\delta\right|_{\nball{2}(0;r_0)} \; \; \; \text{\em for each $\delta\in(0,\delta_0)$}.
 \end{equation}
\end{itemize}
From (C) above, Part~$2a)$ of this proposition follows. If we define:
\begin{align}
 \clcone^k_j \ &:= \ \smcone{\xi_j}{k\al}, \; \; k=1,2, \notag \\
 U_{j,\delta} \ &:= \ u_j-\delta h_j, \notag
\end{align}
for each $j=1,\dots,N$, then, in view of \eqref{E:Gamma} and \eqref{E:agree}, Parts~$2b)$ and 
$2c)$ also follow. The next step of this proof is to
shrink $r_0$, to the extent necessary, to obtain an 
$r_\delta>0$ so that Parts~$2d)$ and $2e)$ follow.
\smallskip

\noindent{{\bf Step~2.} {\em Estimates on the size of $(G_\delta-\rho)$}}

\noindent{For the same reasons as described in Step~1, we shall argue as though
each $L_j$ is of the form $\{z:z_1=\xi_j z_2\}$ for some $\xi_j\in \cplx$. Analogous
arguments will follow if, for some $j^0\leq N$, $L_{j^0}=\{z:z_2=0\}$. So, let us fix
a $j\leq N$ and write
\begin{align}
 P_n(z) \ &= \ P_n((z_1-\xi_j z_2)+\xi_j z_2,z_2) \notag \\
	&\equiv
	\sum_{|\bmew|+|\bnew|=n}C_{\bmew\bnew}(j,n)
		(z_1-\xi_j z_2)^{\eta_1}(\zbar_1-\xibar_j \zbar_2)^{\eta_2}
		z_2^{\nu_1}\zbar^{\nu_2}, \notag 
\end{align}
for each $n=2k+1,\dots,2M_j$, and $j=1,\dots,N$. Here $\bmew$ and $\bnew$ denote
multi-indices. We clarify some notation
\begin{align}
 \bmew &=  (\eta_1,\eta_2), \qquad\qquad |\bmew| = \eta_1+\eta_2 =: \eta, \notag \\ 
 \bnew &=  (\nu_1,\nu_2), \qquad\qquad |\bnew| = \nu_1+\nu_2 =: \nu, \notag \\ 
 \mathfrak{supp}(j,n) &= \{(\eta_1,\eta_2,\nu_1,\nu_2)\in \nat^4:C_{\bmew\bnew}(j,n)\neq 0\}. \notag
\end{align}
Recall that, owing to the change of coordinate described in Step~1, 
$\left. P_n\right|_{L_j}\equiv 0$ for $n=2k+1,\dots,2M_j-1$. Thus, 
we can conclude the following facts:
\begin{itemize}
 \item For $n=2k+1,\dots, 2M_j-1$, the condition \eqref{E:simp} tells us that
 $(\eta_1,\eta_2,\nu_1,\nu_2)\in \mathfrak{supp}(j,n) \, \Longrightarrow \, \eta\geq \mu_j$.
 \item $u_j(s)=\sum_{(0,\bnew)\in  \mathfrak{supp}(j,2M_j)}
 C_{0\bnew}(j,2M_j)s^{\nu_1}\overline{s}^{\nu_2}$.
\end{itemize}
Hence, we get the following estimate
\begin{align}
 (\bmew,\bnew) \, \in \, &\mathfrak{supp}(j,n), \; 2k+1\leq n\leq 2M_j \; 
 \text{\em and} \; \eta\geq 1 \notag \\
 \Longrightarrow \ &\left|C_{\bmew\bnew}(j,n)
                (z_1-\xi_j z_2)^{\eta_1}(\zbar_1-\xibar_j \zbar_2)^{\eta_2}
                z_2^{\nu_1}\zbar^{\nu_2}\right| \notag \\
 &\leq \ \alpha^{\eta-\mu_j}|C_{\bmew\bnew}(j,n)||z_1-\xi_jz_2|^{\mu_j}|z_2|^{n-\mu_j} \; \; \;
	\forall z\in \smcone{\xi_j}{\al}.
		\label{E:junkEst1}
\end{align}
Furthermore, if we write
\[
 \rho(z) \ = \ \sum_{n=2k}^{2M_j}P_n(z)+\tail_j(z)
\]
for each $j=1,\dots,N$, then there exists a constant $K_3>0$ such that
\begin{equation}\label{E:junkEst2}
 |\tail_j(z)| \ \leq \ K_3\|z\|^{2M_j+1} \; \; \forall z\in \nball{2}(0;1).
\end{equation}
Then, in view of the estimates \eqref{E:sizej}, \eqref{E:sizej2}, \eqref{E:junkEst1} and
\eqref{E:junkEst2}, there is a large constant $K_4>0$ such that
\begin{align}
 (G_\delta-\rho)(z) \ \leq \ &-\delta\left(c_1|z_1-\xi_j z_2|^{\mu_j}|z_2|^{2k-\mu_j}
				+c_2|z_2|^{2M_j}\right) \notag \\
 	&\qquad+K_4\sum_{j=2k+1}^{2M_j}|z_1-\xi_j z_2|^{\mu_j}|z_2|^{n-\mu_j}+K_3\|z\|^{2M_j+1} \notag \\
	&\qquad\qquad\forall z\in \smcone{\xi_j}{\al}\bigcap\nball{2}(0;1), \notag
\end{align}
for each $j=1,\dots,N$. Since $n>2k$ in every occurrence of $n$ in the above inequality,
we can find an $r_\delta>0$ sufficiently small and a constant $C>0$ (independent of
$\delta$) such that
\begin{align}
 (G_\delta-\rho)(z) \ \leq \ -\delta C(|z_1-\xi_j z_2|^{\mu_j}|z_2|^{2k-\mu_j}+&|z_2|^{2M_j}) \notag \\
                &\forall (z_1,z_2)\in \smcone{\xi_j}{\al}\bigcap \nball{2}(0;r_\delta), \notag
\end{align}
for each $j=1,\dots,N$ and for each $\delta\in (0,\delta_0)$. This establishes Part~$2d)$.
As for Part~$2e)$, the recipe given by (A)-(C) above implies that 
$(G_\delta-\rho)(z)$ is dominated by a strictly negative function that is homogeneous
of degree $2k$ (see \cite[pp.~53-54]{bharaliStensones:psb09} for details) in
$\left(\CC\setminus \cup_{j=1}^N\smcone{\xi_j}{\al}\right)$, {\em provided $\|z\|$ is
small}. Thus, by arguments very similar to the preceding one, Part~$2e)$ follows.}
\end{proof}

\section{The proof of Theorem~\ref{T:main}}\label{S:main}
 
We begin by stating that several parts of this proof are based on ideas in the proof of Main 
Theorem~1 of \cite{bharaliStensones:psb09}. Define $\nu:=\lcm(m_1,m_2)$ (i.e. the least common 
multiple of $m_1$ and $m_2$) and write $\sigma_j:=\nu/m_j, \ j=1,2.$ Next, define the two 
proper holomorphic maps
\begin{align}
 \psi:\CCC\lrarw\CCC, \; \; \; \psi(s,t_1,t_2)&:=(s,t_1^{\sigma_1},t_2^{\sigma_2}), \notag \\
 \varPsi:\CC\lrarw\CC, \; \; \; \varPsi(t_1,t_2)&:=(t_1^{\sigma_1},t_2^{\sigma_2}). \notag
\end{align}
Finally, write $\varphi:=(P+Q)$ and $R:=\varphi\circ\varPsi$. Also, write
$D:=\{(w,z)\in V_\zt:\er{w}+\varphi(z)<0\}$. The following are easy to verify, and we
shall not dwell on the details:
\begin{itemize}
 \item $R$ is plurisubharmonic.
 \item Let
 \[
  \sum_{n=m_0}^\infty \Pi_{n,R}(t_1,t_2)
 \]
 denote the Taylor expansion of $R$ around $t=0$, with each
 $\Pi_{n,R}$ being the sum of all monomials having total degree $n$. Then, $m_0=\nu$.
 \item The domain $\psi^{-1}(D)$ is given by
 \[
  \omega_R \ := \ \{(s,t)\in \psi^{-1}(V_\zt):\er{s}+R(t)<0\}.
 \]
 \item Write $\Pi:=\Pi_{\nu,R}$. Then, $\Pi$ satisfies all the properties satisfied
 by $P:=P_{2k,R}$ of Proposition~\ref{P:homoBump} (refer to
 \cite[pp. 59-60]{bharaliStensones:psb09} for details).
 \item Furthermore, the domain $\omega_R$ satisfies all the hypotheses of 
 Proposition~\ref{P:homoBump}.
\end{itemize}
It would be useful to look more closely at the second assertion. We state the
following general fact for later use:
\smallskip

\begin{fact}\label{F:weights}
Let $\varphi$ be any $\smoo^\infty$-smooth function defined around $0\in \CC$, and
let $R:=\varphi\circ\varPsi$, where $\varPsi$ is as defined above. Let
$\Pi_{n,R}, \ n\in \nat$, be as defined above and let
\[
 \sum_\eta\mathcal{Q}_{\eta,\varphi}
\]
be a formal rearrangement of the terms of the Taylor expansion of $\varphi$ around $z=0$
such that 
$\mathcal{Q}_{\eta,\varphi}(r^{1/m_1}z_1,r^{1/m_2}z_2)=r^\eta\mathcal{Q}_{\eta,\varphi}(z_1,z_2) 
\ \forall r>0$. If $\Pi_{n,R}=\mathcal{Q}_{\eta,\varphi}\circ\varPsi$, then $n=\nu\eta$.
\end{fact}

\noindent{To see this simple fact, we fix a $t\in \CC$ such that $\Pi_{n,R}(t)\neq 0$. Thus, 
by definition
\[
 r^n\Pi_{n,R}(t) \ = \ \Pi_{n,R}(rt) \ = \ 
	\mathcal{Q}_{\eta,\varphi}(r^{\sigma_1}t_1^{\sigma_1},r^{\sigma_2}t_2^{\sigma_2}) \; \;
							\forall r>0.
\]
On the other hand
\[
 \mathcal{Q}_{\eta,\varphi}(r^{\sigma_1}t_1^{\sigma_1},r^{\sigma_2}t_2^{\sigma_2}) \
 = \ \mathcal{Q}_{\eta,\varphi}((r^\nu)^{1/m_1}t_1^{\sigma_1},(r^\nu)^{1/m_2}t_2^{\sigma_2})
 = \ r^{\nu\eta}\mathcal{Q}_{\eta,\varphi}\circ\varPsi(t) \; \; \forall r>0.
\]
Comparing the two equations above, we conclude that $n=\nu\eta$. The second assertion above
is just a special case of Fact~\ref{F:weights}.}
\smallskip

As in the proof of Proposition~\ref{P:homoBump}, {\em in the interests of brevity, 
we shall frame our arguments as though each curve $X\in \exepc(P)$ is of the form
\[
 \left\{(z_1,z_2):z_1^{m_1/\gcd(m_1,m_2)}=\xi z_2^{m_2/\gcd(m_1,m_2)}\right\}
\]
for some $\xi\in \cplx$, and the reasons are the same as for Proposition~\ref{P:homoBump}
(although see the remark following \eqref{E:veej}).}
It would be useful to look at the fourth assertion above in greater depth. First observe that
\[
 \frac{m_1}{\gcd(m_1,m_2)} \ = \ \sigma_2, \quad \frac{m_2}{\gcd(m_1,m_2)} \ = \ \sigma_1,
\]
and hence note that for any $\xi \in\cplx$ such that
$\{(z_1,z_2):z_1^{\sigma_2}=\xi z_2^{\sigma_1}\}\in\exepc(P)$, $\Pi$ is forced to be harmonic along
each of the complex lines that make up the set  
\[
 \cplxlns(\xi) \ := \ \bigcup_{l=0}^{\sigma_1\sigma_2-1}
 \left\{(t_1,t_2):t_1=|\xi|^{1/\sigma_1\sigma_2}\exp\left(\frac{2\pi il+i{\sf Arg}(\xi)}
 {\sigma_1\sigma_2}\right)t_2\right\}
\]
(here ${\sf Arg}$ denotes some branch of the argument). But since Proposition~\ref{P:homoBump}
is applicable to the domain $\omega_R$, there must be only finitely many complex lines of the 
above description. This implies Part~(1) of our theorem.
\smallskip   

Let us introduce some notation that we shall require. By Part~(1), and in view
of the italicised remark above, let us label the irreducible curves belonging
to $\exepc(P)$ as $X_1,\dots,X_N$, and let us assume that
\[
 X_j \ = \ \left\{(z_1,z_2):z_1^{m_1/\gcd(m_1,m_2)}=\xi_j z_2^{m_2/\gcd(m_1,m_2)}\right\}
\]
for some $\xi_j\in \cplx, \ j=1,\dots,N$. For each $j\leq N$, let us denote the 
collection of complex lines in the set $\cplxlns(\xi_j)$, as defined above, by
$\{L_{j0},L_{j1},\dots,L_{j(\sigma_1\sigma_2-1)}\}$. In keeping with this numbering scheme,
let $\mu_{jk}\in \zahl_+$ denote the numbers determined by \eqref{E:simp} for the pair $(\Pi,R)$.
Similarly, let $M_{jk}\in \zahl_+$ be the integer associated to the complex line 
$L_{jk}\in \exepc(\Pi)$ that is provided by the Proposition~\ref{P:homoBump}. Let
$\varDelta:=\max\{2M_{jk}:1\leq j\leq N, \ 0\leq k\leq\sigma_1\sigma_2-1\}$.
With these notations, we are ready to construct the desired functions.
\smallskip

\noindent{{\bf Step~1.} {\em Constructing $G$ and the auxiliary functions}}

\noindent{Consider the unitary transformations $R^{lm}:(z_1,z_2)\longmapsto 
(e^{2\pi il/\sigma_1}t_1,e^{2\pi im/\sigma_2}t_2)$, $l,m\in \zahl$. Recall that
if $\rho$ is as given by Proposition~\ref{P:homoBump} when applied to the domain
$\omega_R$, then $\rho$ is obtained by subtracting from $R$ all pluriharmonic terms of 
degree\,$\leq \varDelta$ occurring in the Taylor expansion of $R$ around $u=0$. By construction
\begin{align}
 \rho\circ R^{lm}(t) \ &= \ \rho(t) \quad \forall t\in \CC, \; \; \forall l,m\in \zahl, \label{E:inv1} \\
 \Pi\circ R^{lm}(t) \ &= \ \Pi(t) \quad \forall t\in \CC, \; \; \forall l,m\in \zahl. \label{E:inv2}
\end{align}
From the transformation properties of the Levi form, we conclude from \eqref{E:inv2} that
\[
 \levi{\Pi}(R^{lm}t;R^{lm}V) \ = \ \levi{\Pi}(t;V) \quad
 \forall (t,V)\in \CC\times\CC, \; \; \forall l,m\in \zahl. \notag
\]
Hence, if we define $\mathfrak{N}_{\Pi}(t)$ to be the null-space of $\levi{\Pi}(t;\bcdot)$, then
the above statement reveals that:
\begin{itemize}
 \item[$(i)$] With $\LeviDeg{\Pi}$ as defined in Section~\ref{S:results}, $t\in \LeviDeg{\Pi}$ and
 $V\in \mathfrak{N}_{\Pi}(t)$ if and only if  $R^{lm}(t)\in \LeviDeg{\Pi}$ and  
 $R^{lm}(V) \in\mathfrak{N}_{\Pi}(R^{lm}(t))$, $l,m\in \zahl$.
\end{itemize}
If, for each triple $(k,l,m)$ with $k,l,m$ belonging to the respective integer-ranges 
established above, we define $\varkappa(k,l,m)$ by the relation
\[
 0\leq \varkappa(l,m,n)\leq \sigma_1\sigma_2-1 \; \; \text{\em and} \; \; 
	\varkappa(l,m,n)\,\equiv\,(k+\sigma_2l-\sigma_1m)\,{\rm mod}(\sigma_1\sigma2),
\]
then 
\begin{equation}\label{E:lineTrans}
 R^{lm}(L_{jk}) \ = \ L_{j,\varkappa(k,l,m)}.  
\end{equation}
Now note that since $\gcd(\sigma_1,\sigma_2)=1$,
\begin{equation}\label{E:gcd}
 \text{\em For each $k=0,\dots,\sigma_1\sigma_2-1$, $\exists l,m\in \zahl$ such that
	$k=\sigma_2l-\sigma_1m$.}
\end{equation}
From the observation $(i)$, the prescription provided by Proposition~\ref{P:leadBump} for 
constructing each $H_{jk}$ associated
to a complex line $L_{jk}$, and from \eqref{E:lineTrans}, it is clear that
in the construction of the $\gee_{\delta}$ given by the equation \eqref{E:gee}:
\begin{itemize}
 \item[$(ii)$] For each $k$, we may set $\clcone^n_{jk}=R^{lm}\left(\clcone^n_{j0}\right)$,
 $n=1,2$, taking appropriate $l,m\in \zahl$;
 \item[$(iii)$] For each $k$, we may set $H_{jk}=H_{j0}\circ(R^{lm})^{-1}$ and 
 $\Psi_{jk}=\Psi_{j0}\circ(R^{lm})^{-1}$, taking appropriate $l,m\in \zahl$;
\end{itemize}
and the conclusions of Proposition~\ref{P:homoBump} will continue to hold true.
That we can {\em always} find appropriate $l,m\in \zahl$ for the purposes
of $(ii)$ and $(iii)$ follows from \eqref{E:gcd}.}
\smallskip

Result~\ref{R:noell} is applicable to $\Pi$. A careful examination of its proof
reveals that Noell's construction of the bumping is local. Hence, in view of \eqref{E:inv2}
and $(i)$, we can construct the function $H_0$ that we use in the proof of Proposition~\ref{P:homoBump},
as well as the perturbation described in (A) and (B) in the proof of Proposition~\ref{P:homoBump} in 
such a way that:
\begin{equation}\label{E:HNew}
 H\circ R^{lm}(t) \ = \ H(t) \; \; \forall u\in \CC \; \text{\em and} \;
 \forall l,m\in \zahl.
\end{equation}
And finally, owing to \eqref{E:inv1}:
\begin{align}
 M_{jk} \ &= \ M_{jk^*} \ =: \ M_j \; \; \text{\em for any $k\neq k^*$}, \notag \\ 
 u_{jk}(x) \ &= \ \Pi_{2M_j}\left(|\xi_j|^{1/\sigma_1\sigma_2}\exp\left(\frac{2\pi ik+i{\sf Arg}(\xi_j)}
 {\sigma_1\sigma_2}\right)x,x\right), \; \; \; \forall x\in \cplx,
 \label{E:subhRels}
\end{align} 
where $\ k=0,\dots,\sigma_1\sigma_2-1$. 
\smallskip

Recall that the change of coordinate mentioned in Proposition~\ref{P:homoBump} is represented by a 
biholomorphism of the form
\[
 \tau_1(s,t_1,t_2) \ := \ (s-q(t_1,t_2),t_1,t_2), \; \; (s,t_1,t_2)\in \CCC,
\]
where $q$ is the holomorphic polynomial that is the sum of all monomials in 
$t_1$ and $t_2$ whose real and imaginary parts are the pluriharmonic terms 
of some $\Pi_{n,R}$, $n=\nu+1,\dots,\varDelta:=\max_{j,k}2M_{jk}$. However, since
$R=\varphi\circ\varPsi$, {\em every monomial in $q$ is the product of integer powers 
of $t_1^{\sigma_1}$ and $t_2^{\sigma^2}$} (in fact, this observation is the implicit reason
for the the assertion \eqref{E:inv1} above). Thus, a diagram-chase reveals that 
there is a biholomorphism $\tau_2:\CCC\lrarw\CCC$ that makes the following
diagram commute:
\[
\begin{CD}
\CCC 			@>\psi>> 	\CCC \\
@V\tau_1=(s-q(t),t)VV			@VV\tau_2V \\
\CCC			@>>\psi>	\CCC
\end{CD}.
\]
Furthermore $\tau_2$ is of the form
\begin{equation}\label{E:tau2}
 \tau_2(w,z_1,z_2) \ = \ (w-f(z_1,z_2),z_1,z_2), \; \, (w,z_1,z_2)\in \CCC,
\end{equation}
where $f$ is a polynomial in $z_1$ and $z_2$. Let us write
$(\what,z_1,z_2):=\tau_2(w,z_1,z_2)$, and let 
\begin{equation}\label{E:intermRep}
 \OM\bigcap U_\zt \ = \ \{(\what,z):\er{\what}+\newR(z)+\newRem(\mi{\what},z)<0\}
\end{equation}
denote the local representation of $\OM$ relative to the new coordinate system
$(U_\zt;\what,z_1,z_2)$. Then, it is clear that
\[
 \psi\circ\tau_1(\omega_R) \ = \ \{(\what,z)\in \tau_2(U_\zt):\er{\what}+\newR(z)<0\}.
\]
We now pick and fix a value in $(0,\delta_0)$, say $\delta_0/2$, and use this to define the 
functions mentioned in the statement of our theorem. To this end, we define:
\begin{align}
 G(z) \
	&:= \ \frac{1}{\sigma_1\sigma_2}\sum_{a=1}^{\sigma_1}\sum_{b=1}^{\sigma_2}
	G_{\delta_0/2}\left(|z_1|^{1/\sigma_1}\exp\left(\frac{2\pi ia+i{\sf Arg}(z_1)}{\sigma_1}\right)
	\right., \label{E:G} \\   
	&\qquad\quad
	\left.|z_2|^{1/\sigma_2}\exp\left(\frac{2\pi ib+i{\sf Arg}(z_2)}{\sigma_2}\right)\right),\notag \\
 \clweg^l_j \ &:= \ \varPsi(\clcone^l_{j0}), \; \; \; l=1,2, \; \ j=1,\dots,N, \label{E:wedgeImp}
\end{align}
where all the objects on the right-hand sides of \eqref{E:G} and \eqref{E:wedgeImp} are as given
in Proposition~\ref{P:homoBump}.
We remark here that, owing to $(ii)$ above, we could as well have used $\clcone^l_{jk}$,
$k=0,\dots,\sigma_1\sigma_2-1$, in the definition \eqref{E:wedgeImp} for each fixed $j$. Let us now 
define
\begin{align}
 \mathcal{H}_0(z) \
 	&:= \ \frac{\delta_0}{2\sigma_1\sigma_2}\sum_{a=1}^{\sigma_1}\sum_{b=1}^{\sigma_2}
	H\left(|z_1|^{1/\sigma_1}\exp\left(\frac{2\pi ia+i{\sf Arg}(z_1)}{\sigma_1}\right)
	\right., \notag \\
	&\qquad\quad
	\left.|z_2|^{1/\sigma_2}\exp\left(\frac{2\pi ib+i{\sf Arg}(z_2)}{\sigma_2}\right)\right).\notag
\end{align}
As $H$ is homogeneous of degree $\nu$, $\mathcal{H}_0$ is $(m_1,m_2)$-homogeneous. The other
properties of $\mathcal{H}_0$ listed in Theorem~\ref{T:main} are immediate. Furthermore, note that:
\begin{align}
 P(z) \
        &= \ \frac{1}{\sigma_1\sigma_2}\sum_{a=1}^{\sigma_1}\sum_{b=1}^{\sigma_2}
        \Pi\left(|z_1|^{1/\sigma_1}\exp\left(\frac{2\pi ia+i{\sf Arg}(z_1)}{\sigma_1}\right)
        \right., \notag \\
        &\qquad\quad
        \left.|z_2|^{1/\sigma_2}\exp\left(\frac{2\pi ib+i{\sf Arg}(z_2)}{\sigma_2}\right)\right).
	\label{E:P}
\end{align}
Comparing the equations \eqref{E:G} and \eqref{E:P} with our definition of
$\mathcal{H}_0$, and keeping \eqref{E:wedgeImp} in mind, we already have the first and the
third properties listed at the end of Theorem~\ref{T:main}. The task of defining the
$v_j$'s is subtler. If we would like to recover the second property listed at the
end of Theorem~\ref{T:main}, we would first need to study the quantity
$\left.(G-P+\mathcal{H}_0)\right|_{\clweg^1_j}$. Accordingly, we first fix a $J\leq N$, 
then refer to the formula for $\widetilde{\Gamma}_{\delta_0/2}$ from 
\eqref{E:Gamma}, whence {\em for an $R>0$ sufficiently small}, we get:
\begin{align}
 (G-&P+\mathcal{H}_0)|_{{\clweg^1_J}} \notag \\
 	= \ &\frac{1}{\sigma_1\sigma_2}\sum_{k=0}^{\sigma_1\sigma_2-1}
	\sum_{a=1}^{\sigma_1}\sum_{b=1}^{\sigma_2}
        \Psi_{Jk}\left(|z_1|^{1/\sigma_1}\exp\left(\frac{2\pi ia+i{\sf Arg}(z_1)}{\sigma_1}\right)
        \right., \notag \\
        &\left.|z_2|^{1/\sigma_2}\exp\left(\frac{2\pi ib+i{\sf Arg}(z_2)}{\sigma_2}\right)\right)
	U_{Jk,\delta_0/2}
	\left(|z_2|^{1/\sigma_2}\exp\left(\frac{2\pi ib+i{\sf Arg}(z_2)}{\sigma_2}\right)\right)
        \notag \\
 	&\qquad\qquad\qquad \forall z\in\clweg^1_J\bigcap\nball{2}(0;R). \label{E:defv1}
\end{align}
We refer the reader to the proof of Proposition~\ref{P:homoBump}, and recall the numbering
scheme introduced above, for the meaning of the quantities $U_{jk,\delta_0/2}$ in 
\eqref{E:defv1}. Following \eqref{E:Gamma} strictly, there should be a double-sum
over all possible values of $(j,k)\in\{1,\dots,N\}\times\{0,\dots,\sigma_1\sigma_2-1\}$.
It is easy to see, however, that for $z\in \clweg^1_J$, the summands involving
the indices $j\neq J$ vanish. Note the following points that follow by construction:
\begin{itemize}
 \item Since, by definition, for each $j=1,\dots,N$, $\Pi_{2M_j}$ is the sum of monomials 
 of degree $2M_j$ that are products of integer powers of $t_l^{\sigma_l}$ and 
 $\overline{t}_l^{\sigma_l}$, $l=1,2$, the equation \eqref{E:subhRels} tells us that
 the second factor in the expression \eqref{E:defv1} does not change as $b$ ranges through
 the set $\{1,\dots,\sigma_2\}$.
 \item In view of $(ii)$, $(iii)$ and \eqref{E:wedgeImp},
 the first factor in \eqref{E:defv1} does not change as 
 $k$ ranges through the set $\{0,\dots,\sigma_1\sigma_2-1\}$.
\end{itemize}
In view of these two points, the equation \eqref{E:defv1} reveals that the
second property listed at the end of Theorem~\ref{T:main} is established if we define
\begin{equation}\label{E:veej}
 v_j(x) \ := \ 
 \frac{1}{\sigma_1\sigma_2}\sum_{k=0}^{\sigma_1\sigma_2-1}\left[
        \sum_{b=1}^{\sigma_2}U_{jk,\delta_0/2} 
        \left(|x|^{1/\sigma_2}\exp\left(\frac{2\pi ib+i{\sf Arg}(x)}{\sigma_2}\right)\right)\right]
	\; \; \forall x\in \cplx,
\end{equation}
for each $j=1,\dots,N$. The second property is established because, for a fixed
$J\leq N$ and $z\in\clweg^1_J$, the right-hand side of the above equation (with
$x$ replaced by $z_2$) is just a different way of expressing the right-hand side of
\eqref{E:defv1}. We claim that each $v_j$ is subharmonic. To see this, we note that the
functions within the square brackets in \eqref{E:veej} are subharmonic. This follows from
the following general fact:
\smallskip

\begin{fact}\label{F:properPush}
Let $D_1$ and $D_2$ be two domains in $\Cn$ and let
$p:D_1\lrarw D_2$ be a proper holomorphic map. If $u$ is a plurisubharmonic function
on $D_1$, then the function 
\[
 v(z) \ := \ \sum_{t\in p^{-1}\{z\}}u(t), \; \; z\in D_2,
\]
is plurisubharmonic on $D_2$.
\end{fact}

\noindent{Therefore $v_j$, being a positive linear combination of subharmonic functions, is 
subharmonic. In fact, as each $U_{jk,\delta}$ is, by construction, {\em strictly}
subharmonic away from $x=0$, and since the map $x\longmapsto x^{\sigma_2}$ is a
local biholomorphism away fron $x=0$, each $v_j$ is, in fact, strictly subharmonic
on $\cplx\setminus\{0\}$. {\em {\bf{\em Note that}} if for some $j\leq N$, say $j^0$, $X_{j^0}$ is
the $z_1$-axis}, then --- in view of the work done in Proposition~\ref{P:homoBump} --- 
there would be just a single $U_{j^0k,\delta_0/2}=:U_{j^0,\delta_0/2}$. Accordingly,
the formulas \eqref{E:G} and \eqref{E:veej} would reflect the following changes:
\begin{itemize}
 \item In defining $v_{j^0}$, the $U_{jk,\delta_0/2}$'s in \eqref{E:veej} would be 
 replaced by $U_{j^0,\delta_0/2}$ and every occurrance of $\sigma_2$ would be 
 replaced by $\sigma_1$.
 \item In \eqref{E:G}, $v_{j^0}$ would appear as a term depending on $z_1$.
\end{itemize}}
\smallskip

Since $G_{\delta_0/2}$ is, by construction, plurisubharmonic on $\CC$, we see from 
Fact~\ref{F:properPush}, and the formula \eqref{E:G}, that $G$ is plurisubharmonic on $\CC$.
At this stage, we only need to establish \eqref{E:bumped}. This will require one more
holomorphic change of coordinate.
\smallskip

\noindent{{\bf Step 2.} {\em Establishing the relation \eqref{E:bumped}}}

\noindent{Let $R_1>0$ be so small that $\nball{3}(0;R_1)\subset V_\zt$. Shrinking $R_1$ 
further if necessary, it follows from our definition \eqref{E:wedgeImp} of the $(m_1,m_2)$-wedges 
$\clweg^1_j$, $j=1,\dots,N$, and from Part~$2e)$ of Proposition~\ref{P:homoBump} (applied to
the domain $\omega_R$) that there is a constant $A_1>0$ such that
\begin{equation}\label{E:bound1}
 (G-\newR)(z) \ \leq \ -A_1(|z_1|^{m_1}+|z_2|^{m_2}) \; \; \; 
 \forall (z_1,z_2)\in \left(\CC\setminus \cup_{j=1}^N\clweg^1_j\right)\bigcap\nball{2}(0;R_1).
\end{equation}
Recall that $\newR$ is as given by \eqref{E:intermRep}. Let us now fix a $j\leq N$. Now, obviously:
\begin{align}
 (G-\newR)(z) \
        &:= \ \frac{1}{\sigma_1\sigma_2}\sum_{a=1}^{\sigma_1}\sum_{b=1}^{\sigma_2}
        (G_{\delta_0/2}-\rho)
	\left(|z_1|^{1/\sigma_1}\exp\left(\frac{2\pi ia+i{\sf Arg}(z_1)}{\sigma_1}\right)
        \right., \notag \\
        &\qquad\quad
        \left.|z_2|^{1/\sigma_2}\exp\left(\frac{2\pi ib+i{\sf Arg}(z_2)}{\sigma_2}\right)\right) 
	\notag \\
	&\equiv \frac{1}{\sigma_1\sigma_2}\sum_{a=1}^{\sigma_1}\sum_{b=1}^{\sigma_2}Y(z;a,b).
        \label{E:crudeDef}
\end{align}
Note that if $z\in \clweg^1_j$, then, from our discussions above, the arguments of the
functions constituting the sum above would fall into one of the cones
$\clcone^1_{jk}$, $k=0,\dots,\sigma_1\sigma_2-1$. Applying Part~$2d)$ of Proposition~\ref{P:homoBump},
we get
\begin{align}
 &\text{\em $z\in \clweg^1_j$ is such that 
 $\left(|z_1|^{1/\sigma_1}e^{i\sfrac{2\pi a+{\sf Arg}(z_1)}{\sigma_1}},
	|z_2|^{1/\sigma_2}e^{i\sfrac{2\pi b+{\sf Arg}(z_2)}{\sigma_2}}\right)$ lies} \notag \\
 &\text{\em in
	$\clcone^1_{jk}\bigcap\nball{2}(0;r_{\delta_0/2})$} \notag \\
 &\quad\Longrightarrow \
	Y(z;a,b) \notag \\
 &\quad \leq \ -\frac{C\delta_0}{2}
	\left\{\left||z_1|^{1/\sigma_1}e^{i\sfrac{2\pi a+{\sf Arg}(z_1)}{\sigma_1}}
 -\xi_{jk} |z_2|^{1/\sigma_2}e^{i\sfrac{2\pi b+{\sf Arg}(z_2)}{\sigma_2}}\right|^{\mu_j}
 |z_2|^{\sfrac{\nu-\mu_j}{\sigma_2}}+|z_2|^{\sfrac{2M_j}{\sigma_2}}\right\}, 
 \label{E:finer1}
\end{align}
where $\xi_{jk}:=
|\xi_{j}|^{1/\sigma_1\sigma_2}\exp\left(i\frac{2\pi k+{\sf Arg}(\xi_j)}{\sigma_1\sigma_2}\right)$,
and where we remind the reader that $\mu_{j0}=\dots=\mu_{j(\sigma_1\sigma_2-1)}=:\mu_j$ and
$M_{j0}=\dots=M_{j(\sigma_1\sigma_2-1)}=:M_j$. Thus, if $R_2>0$ is so small that
the the arguments of the
functions constituting the right-hand side of \eqref{E:crudeDef} are in $\nball{2}(0;r_{\delta_0/2})$,
then, from \eqref{E:crudeDef} and \eqref{E:finer1}, we have a non-negative function 
$\mathcal{N}_j:\nball{2}(0;R_2)\lrarw [0,\infty)$ such that
\begin{equation}\label{E:finer2}
 (G-\newR)(z) \ \leq \ -\frac{C\delta_0}{2}\mathcal{N}_j(z) \; \; \; 
	 \forall (z_1,z_2)\in \clweg^1_j\bigcap\nball{2}(0;R_2).
\end{equation}
Now, the whole point of the estimate in $2d)$ (applied to the domain $\omega_R$)
above was that there exists a constant
$A_2>0$ such that
\[
 |\rho(t)| \ \leq \ A_2(|t_1-\xi_{jk} t_2|^{\mu_j}|t_2|^{\nu-\mu_j}+|t_2|^{2M_j}) \; \; \;
	\forall t\in \clcone^1_{jk}\bigcap{\sf Dom}(\rho),
\]
for all relevant $(j,k)$. Hence, if $z\in \clweg^1_j$ satisfies the hypothesis of the
statement \eqref{E:finer1}, then $|\newR(z)|$ has as upper bound the quantity obtained
from the right-hand side of \eqref{E:finer1} with the factor $-(C\delta_0/2)$ replaced by $A_2$.
Eventually, therefore, we get the estimate
\begin{equation}\label{E:finer3}
 |\newR(z)| \ \leq \ A_2\mathcal{N}_j(z) \; \; \;
         \forall (z_1,z_2)\in \clweg^1_j\bigcap\nball{2}(0;R_2).
\end{equation}}
\smallskip

We are now ready to establish \eqref{E:bumped}. To this end, we define new coordinates
$(W,Z_1,Z_2)$ as follows
\[
 W \ := \ \what+K\what^2, \qquad\quad Z_l \ := \ z_l, \; \; l=1,2,
\]
where $K>0$. The mapping $(\what,z)\longmapsto (W,Z)$ is invertible in a neighbourhood
of $0\in \CCC$, whose size depends on the parameter $K$. We will choose a suitable $K>0$
so that \eqref{E:bumped} is achieved. Recall, from \eqref{E:intermRep}, the 
relation $\newRem(\mi{\what},z)=r(\mi{w},z)$. Let us write 
$\er{\what}=:\alpha$ and $\mi{\what}=:\beta$. In view of $(**)$, there exists an $A_3>0$ 
such that
\begin{align}
 |\newRem(\beta,z)| \ = \ |r(\mi{w},z)| \ &\leq \ A_3|\mi(\what+f(z))|^q \notag \\
					&\leq \ A_3(|\beta|^q+|\mi(f(z))|^q) \; \; 	  
 \forall (\what,z)\in \overline{\nball{3}(0;R_1/2)}, \label{E:rEst}
\end{align}
where $q:=\varDelta_1(\bdy\OM)+(1/\nu)$ by hypothesis, and
where $f$ is the polynomial described in \eqref{E:tau2}. Note that if we decompose $f$ as
\[
 f \ = \ \sum_{\eta}\mathcal{Q}_{\eta,f},
\]
where the right-hand side is as described in the statement of Fact~\ref{F:weights}, and
if each $\mathcal{Q}_{\eta,f}\circ\varPsi$ is the sum of pluriharmonic terms contained
in some $\Pi_{n,R}$, then $\eta\geq \nu+1$. This is because $\Pi_{\nu,R}$ itself contains
no pluriharmonic terms. Then, in view of Fact~\ref{F:weights} and \eqref{E:rEst}, 
and raising the value of $A_3>0$ if necessary, we have
\begin{equation}\label{E:fEst}
 |\newRem(\beta,z)| \ \leq \ A_3\left(|\beta|^q+
				(|z_1|^{m_1}+|z_2|^{m_2})^{q(\nu+1)/\nu}\right) \; \; 
				\forall z\in \overline{\nball{2}(0;R_1/2)}.
\end{equation}
By a similar appeal to Fact~\ref{F:weights} (and raising $A_3>0$ even further if necessary):
\begin{equation}\label{E:finer4}
 \mathcal{N}_j(z) \ \geq \ A_3^{-1}(|z_1|^{m_1}+|z_2|^{m_2})^{(\varDelta+1)/\nu} \; \;
         \forall (z_1,z_2)\in \clweg^1_j\bigcap\nball{2}(0;R_2),
\end{equation}
for each $j=1,\dots,N$. Recall that $\varDelta:=\max\{2M_{jk}:j\leq N, \ k=0,\dots,\sigma_1\sigma_2-1\}$.
\smallskip 

Note that we would be done if we could show that:

\begin{center}
\begin{tabular}{p{12.7cm}}
 {\em For each $(\what,z)\neq 0$ such that $(\alpha+\newR(z)+\newRem(\beta,z))=0$,
 provided $|\what|$ and $\|z\|$ are sufficiently small, 
 $(\er(\what+K\what^2)+G(z))<0$.}
\end{tabular}
\end{center}

\noindent{The above statement is true when
$z\in ((\CC\setminus \cup_{j=1}^N\clweg^1_j)\bigcap\nball{2}(0;R_1))$, for $R_1>0$ sufficiently
small, and for a suitable choice of $K>0$. This follows from the estimate \eqref{E:bound1}; it
was shown by Yu \cite[pp. 604-605]{yu:wblgK-Rmwpd95} using an idea of
Fornaess--Sibony \cite{fornaessSibony:cpfwpd89}.}
\smallskip

It thus remains to establish the above statement when $z\in \clweg^1_j$ (provided
$(\what,z)\neq 0$ and sufficiently close to the origin). We fix $j\leq N$ for the 
following calculation and write
\begin{align}
 \er(\what+K\what^2)+G(z) \ &= \ \alpha+K(\alpha^2-\beta^2)+G(z) \notag \\
	&= \ -K\beta^2-(\newR(z)+\newRem(\beta,z))+K(\newR(z)+\newRem(\beta,z))^2+G(z) \notag \\
	&\leq \ -K\beta^2+(G-\newR)(z) + |\newRem(\beta,z)| 
		+2K(\newR(z)^2+\newRem(\beta,z)^2) \notag \\
	&\qquad
	\text{\em provided $\alpha+\newR(z)+\newRem(\beta,z)=0$, $(\what,z)\neq 0$.} \label{E:1stCalc}
\end{align}
Let us write $R:=\min(R_1/2,R_2)$. Then, applying to \eqref{E:1stCalc} the
estimates \eqref{E:finer2}, \eqref{E:finer3} and \eqref{E:fEst}, and shrinking
$R>0$ if necessary so that $2K\newRem(\beta,z)^2\leq (K/2A_3)|\newRem(\beta,z)|$, we get
\begin{align}
 \er(\what+&K\what^2)+G(z) \notag \\
 \leq \ &-K\beta^2+\left[2KA_2^2\mathcal{N}_j^2(z)-\frac{C\delta_0}{2}\mathcal{N}_j(z)\right]
	+ (1+(K/2A_3))|\newRem(\beta,z)| \notag \\
 \leq \ &\left(-\frac{K}{2}+A_3\right)\beta^2 + 
  \left[2KA_2^2\mathcal{N}_j^2(z)-\frac{C\delta_0}{2}\mathcal{N}_j(z)\right] \notag \\
 &+ (A_3+\tfrac{1}{2}K)(|z_1|^{m_1}+|z_2|^{m_2})^{q(\nu+1)/\nu} \quad
	\forall z\in \clweg^1_j\bigcap(\nball{3}(0;R)\setminus\{0\}).
\end{align}
At this stage, {\em we fix $K>0$ once and for all} so that $(-(K/2)+A_3)\leq-1$ and is large
enough for Yu's argument; this determines the final holomorphic coordinates required for 
our theorem. With $K$ fixed, we can shrink
$R>0$ further if necessary and apply \eqref{E:finer4} so that the following inequalities follow
(let us write $A:=A_3+\frac{1}{2}K$, $B:=C\delta_0/4$ and $\Sigma(z):=|z_1|^{m_1}+|z_2|^{m_2}$):
\begin{align}
 \er(\what+K\what^2)+G(z) \
 &\leq \ -\beta^2-B\mathcal{N}_j(z)+A(|z_1|^{m_1}+|z_2|^{m_2})^{q(\nu+1)/\nu} \notag \\
 &\leq \ -\beta^2-(B/A_3)\Sigma(z)^{(\varDelta+1)/\nu}+A\Sigma(z)^{q(\nu+1)/\nu} \notag \\
 &\qquad\qquad
	\forall z\in \clweg^1_j\bigcap(\nball{3}(0;R)\setminus\{0\}). \label{E:2ndCalc}
\end{align}
The curves $X_j$ have order of contact $2M_j/\nu$ with $\bdy\OM$ at the
origin. Hence $\varDelta_1(\bdy\OM)\geq \varDelta/\nu$. This means
\[
 \left(\varDelta_1(\bdy\OM)+\frac{1}{\nu}\right)\frac{\nu+1}{\nu} \ > \ 
 \frac{\nu\varDelta_1(\bdy\OM)+1}{\nu} \ \geq \  \frac{\varDelta+1}{\nu}.
\]
Applying this to \eqref{E:2ndCalc} affords us an $R>0$ sufficiently small that
\[
 \er(\what+K\what^2)+G(z) \ \leq \ -\beta^2-\frac{B}{2A_3}\Sigma(z)^{(\varDelta+1)/\nu} \
 < \ 0 \; \; \forall z\in \clweg^1_j\bigcap(\nball{3}(0;R)\setminus\{0\}),
\]
for each $j=0,\dots,N$. In view of our preceding remarks, the proof is complete. \qed
\smallskip


\end{document}